\documentclass[a4paper,12pt]{article}

\usepackage[pdftex]{graphicx}
\usepackage[section] {placeins}
\usepackage{enumerate}
\usepackage[normalem]{ulem}
\usepackage{amsmath}
\usepackage{amsthm}
\usepackage{latexsym}
\usepackage{amsfonts}
\usepackage{amssymb}
\usepackage{mathtools}
\usepackage{longtable}

\usepackage{tikz-cd}
\usetikzlibrary{graphs}
\usetikzlibrary{arrows.meta}
\usetikzlibrary{shapes.geometric}

\usepackage{authblk}

\usepackage{latexsym,amssymb,amsfonts, amsthm, amsmath}
\usepackage[english]{babel}
\usepackage{bm}
\usepackage{mathrsfs}
\usepackage{algpseudocode}
\usepackage{algorithm}

\algnewcommand{\IIf}[1]{\State\algorithmicif\ #1\ \algorithmicthen}
\algnewcommand{\EndIIf}{\unskip\ \algorithmicend\ \algorithmicif}

\newtheorem{thm}{Theorem}[section]
\newtheorem{lem}[thm]{Lemma}
\newtheorem{cor}[thm]{Corollary}
\newtheorem{prop}[thm]{Proposition}

\newtheorem{ex}[thm]{Example}

\theoremstyle{definition}
\newtheorem{defn}[thm]{Definition}

\newtheorem{rem}[thm]{Remark}

\newcommand{\Z}{\mathbb{Z}}

\def\Sym{\mathrm{Sym}}

\usepackage{fullpage}

\makeatletter
\newcommand{\subjclass}[2][1991]{%
\let\@oldtitle\@title%
\gdef\@title{\@oldtitle\footnotetext{#1 \emph{Mathematics subject classification.} #2}}%
}
\newcommand{\keywords}[1]{%
\let\@@oldtitle\@title%
\gdef\@title{\@@oldtitle\footnotetext{\emph{Keywords and phrases.} #1.}}%
}
\makeatother

\title{Sequencings in Semidirect Products via the Polynomial Method}

\date{}

\author[1]{S.  Costa}
\author[2]{S.  Della Fiore}
\author[3]{M. A. Ollis}

\affil[1]{DICATAM, Sez.~Matematica, Universit\`a degli Studi di Brescia, \newline Via Branze~43, I~25123 Brescia, Italy}
\affil[2]{DI, Universit\`a degli Studi di Salerno, Via Papa Giovanni Paolo II~132, \newline I~84084 Fisciano, Italy, }
\affil[3]{Emerson College, 120 Boylston Street, Boston, MA 02116, USA}
\begin{document}
\maketitle
\begin{abstract}
The {\em partial sums} of a sequence ${\mathbf x} = x_1, x_2, \ldots, x_k$ of distinct non-identity elements of a group~$(G,\cdot)$ are~$s_0 = id_G$ and~$s_j = \prod_{i=1}^j x_i$ for~$0 < j \leq k$.  If the partial sums are all different then~${\mathbf x}$ is a {\em linear sequencing} and if the partial sums are all different when~$|i-j| \leq t$ then~${\mathbf x}$ is a {\em $t$-weak sequencing}.  We investigate these notions of sequenceability in semidirect products using the polynomial method.  We show that every subset of order~$k$ of the non-identity elements of the dihedral group of order~$2m$ has a linear sequencing when~$k \leq 12$ and either~$m>3$ is prime or every prime factor of~$m$ is larger than~$k!$, unless~$s_k$ is unavoidably the identity; that every subset of order~$k$ of a non-abelian group of order three times a prime has a linear sequencing when~$5 < k \leq 10$, unless~$s_k$ is unavoidably the identity; and that if the order of a group is~$pe$ then all sufficiently large subsets of the non-identity elements  are~$t$-weakly sequenceable when~$p>3$ is prime, $e \leq 3$ and $t \leq 6$.
\end{abstract}

\section{Introduction}
Let~$(G,\cdot)$ be a group.  We continue the investigation of when the elements of a subset of~$G$ may be ordered in such a way that their partial sums are distinct.   We set up a general framework for investigating semidirect product groups and look in detail at the dihedral groups of order twice a prime and non-abelian groups of order three times a prime.

\begin{defn}
Let~$S$ be a subset of the non-identity elements of a group~$G$ with~$|S| = k$.  If~$x_1, x_2, \ldots, x_k$ is an ordering of the elements of~$S$, define the $j$th {\em partial sum} of the ordering by~$s_0 = id_G$ and~$s_j = \prod_{i=1}^j x_i$ for~$0 < j \leq k$.  The ordering of the subset~$S$ is a {\em linear sequencing} if the partial sums $s_0,s_1,\dots,s_{k}$ are all different.  It is a {\em cyclic sequencing} if the partial sums are different with the exception that $s_{0}=s_k=id_G$.  If~$S$ has a linear or cyclic sequencing then it is {\em sequenceable}.
\end{defn}

If every subset of the non-identity elements of a group is sequenceable then the group is {\em strongly sequenceable}.  Combining several results and conjectures concering subsets of abelian groups dating back to~1961~\cite{Gordon61}, we have the conjecture that all abelian groups are strongly sequenceable.   See~\cite{AL20,CDOR,PD} for details of the history, motivation and current state of this conjecture. 

The question for non-abelian groups is much more widely open.  In~\cite{CMPP18,Ollis} small counterexamples are presented to show that the  potential conjecture that all non-abelian groups are strongly sequenceable is false.  

{\em Keedwell's Conjecture}~\cite{Keedwell} is that if~$S$ contains all the non-identity elements of a non-abelian group~$G$, then a linear sequencing exists if and only if~$|G| \geq 10$.  This known to hold for many groups, including dihedral groups~\cite{Isbell,Li}, soluble groups with a single involution~\cite{AI}, at least one group of every odd order at which a non-abelian group exists~\cite{OT} and all groups of order at most~255~\cite{OllisSmall}.

Let~$S$ be a subset of size~$k$ of~$G \setminus \{ id_G \}$ whose elements can be ordered so that their total sum is not the identity.  Then~$S$ is known to have a linear sequencing in the following situations~\cite{Ollis}: 
\begin{itemize}
\item $k \leq 4$ and~$S$ does not have one of two specific structures,
\item $G$ is dihedral of order~$2m$ for prime~$m > 3$  and~$k \leq 9$,
\item $G$ is dihedral of order~$2m$ for even or prime $m>4$ and~$k = 2m-2$.
\end{itemize}
In this paper we make two contributions to finding linear sequencings in non-abelian groups.  First, we improve the second item above, showing that it holds up to~$k = 12$ and also for~$m$ whose prime factors are each larger than~$k!$.  (The result that allows for composite~$m$ with large prime factors in fact applies more widely than dihedral groups; see Section~\ref{sec:infinity}.)
The second contribution concerns non-abelian groups of order three times a prime~$p$;  there is at most one such group for any~$p$ and we denote it~$G_{3p}$.  In this case we show that subsets~$S$ of size~$5 < k \leq 10$ are sequenceable and, further, that we can always find a linear sequencing for these subsets when the sum of all the elements is not unavoidably the identity.

In~\cite{CD} a generalization of sequenceability is introduced.

\begin{defn}
Let~$S$ be a subset of the non-identity elements of a group~$G$ with~$|S| = k$ and let~$x_1, x_2, \ldots, x_k$ be an ordering of the elements of~$S$ with partial sums $s_0, s_1, \ldots, s_k$.   If~$s_i \neq s_j$ for any distinct pair of indices~$i,j$ with~$|i-j| \leq t$ then the ordering is a {\em $t$-weak sequencing} of~$S$ and~$S$ is {\em $t$-weakly sequenceable}.
\end{defn}

Observe that a $k$-weak sequencing is exactly a linear sequencing and that a $(k-1)$-weak sequencing is either a linear or cyclic sequencing.  

In~\cite{CD} it is shown that a subset of non-identity elements of the cyclic group~$\Z_p$ of prime order~$p$ is $t$-weak sequenceable when~$t \leq 6$ and, with the additional condition that there are no pairs of the form~$\{x,-x\}$ in the subset, when~$t=7$.  We extend the methodology to apply to groups~$G$ that are not cyclic of prime order, showing that if~$|G| = pe$ then all sufficiently large subsets of~$G \setminus \{ id_G \}$  are~$t$-weakly sequenceable when~$p>3$ is prime, $e \leq 3$ and $t \leq 6$.

In the next section we describe the polynomial method and how it can be implemented for groups that are semidirect products.  In Subsection~\ref{subsec:comp1} and Section~\ref{sec:infinity} we apply the method to obtain the results on sequenceability and in Section~\ref{sec:weak} we apply it to obtain the results on weak sequenceability.

\section{Polynomial Method in Semidirect Products}
In this section, we apply a method that relies on the Non-Vanishing Corollary of the Combinatorial Nullstellensatz, see~\cite{Alon99,Michalek10}. Given a prime $p$ (in the following $p$ will be always assumed to be a prime), this corollary allows us to obtain a nonzero point to suitable polynomials on $\mathbb{Z}_p$ derived starting from the ones defined in \cite{HOS19} and in \cite{Ollis}.
\begin{thm}\label{th:pm}{\rm (Non-Vanishing Corollary)}
Let~$\mathbb{F}$ be a finite field, and let $ f(x_1, x_2, \ldots, x_k)$ be a polynomial in~$\mathbb{F}[x_1, x_2, \ldots, x_k]$. Suppose the degree~$deg(f)$ of~$f$ is $\sum_{i=1}^k \gamma_i$, where each~$\gamma_i$ is a nonnegative integer, and suppose the coefficient of~$\prod_{i=1}^k x_i^{\gamma_i}$ in~$f$ is nonzero. If~$C_1, C_2, \ldots, C_k$ are subsets of~$\mathbb{F}$ with~$|C_i| > \gamma_i$, then there are $c_1 \in C_1, \ldots, c_k \in C_k$ such that~$f(c_1, c_2, \ldots, c_k) \neq 0$.
\end{thm}

In the notation of the Non-Vanishing Corollary, we call the monomial $x_1^{|C_1| - 1}$ $\cdots$ $x_k^{|C_k| - 1}$ the {\em bounding monomial}. The corollary can be rephrased as requiring the polynomial to include a monomial of maximum degree that divides the bounding monomial (where by ``include" we mean that it has a nonzero coefficient).

To use the Non-Vanishing Corollary we require a polynomial for which the nonzeros correspond to successful solutions to the case of the problem under consideration. Here, we will consider subsets $S$ of a group $(G,\cdot_G)$ that is the semidirect product of $(\mathbb{Z}_p,+)$ and another group $(H,\cdot_H)$. For such groups we will use the notation $G=\mathbb{Z}_p\rtimes_{\varphi} H$ for some group homomorphism $\varphi: H \rightarrow Aut(\mathbb{Z}_p)$.
\begin{rem}
These kind of semidirect products define a wide class of groups.
Our hypotheses are, indeed, equivalent to consider groups $(G,\cdot)$ that admit an isomorphic copy of $(\mathbb{Z}_p,+)$ as a normal subgroup $N$ and whose cosets are of the form $N\cdot id,N\cdot c_1,\dots, N\cdot c_{e-1}$ where $e=|G/N|$ and $\{c_0=id,\dots, c_{e-1}\}$ is a subgroup of $G$.
\end{rem}
Now revisit the polynomial method introduced in \cite{Ollis} for dihedral groups and in \cite{CDOR} for groups of type $\mathbb{Z}_p\times H$.
To avoid confusion, from now on, the symbol $+$ will be used only between elements of $\mathbb{Z}_p$.
Let $\pi_2: \Z_p \rtimes_{\varphi} H \rightarrow H$ be the projection map that picks out the second coordinate of an element and for a subset $S \subseteq (\Z_{p} \rtimes_{\varphi} H) \setminus \{ (0,0) \}$ let $\pi_2(S)$ be the multiset $\{ \pi_2(s) : s \in S \}$. Also, the generic element of $\mathbb{Z}_p\rtimes_{\varphi} H$ will be indicated by $x\cdot_G a$ (or $y\cdot b$) implying that $x\in \Z_p$ and $a\in H$. We define the {\em type} of~$S$ to be the sequence $\bm{\lambda} = (\lambda_0, \ldots, \lambda_{e-1} )$, where $\lambda_i$ is the number of times that $i$ appears in $\pi_2(S)$.

Let $T$ be a multiset of elements from $H \setminus \{ 0 \}$ with size~$k$. Let $\bm{a} = (a_1, \ldots, a_k)$ be an arrangement of the elements of~$T$ with partial sums $\bm{b} = (b_0, b_1, \ldots, b_k)$. If, for each~$i$, the element $i$ appears at most $r$ times in~${\bf b}$, then ${\bf a}$ is a {\em quotient sequencing} of~$T$ with respect to~$r$. In our setting, given some~$S \subseteq (\Z_p \rtimes_{\varphi} H) \setminus \{ (0,0) \}$ for which we wish to find a sequencing, we shall be interested in quotient sequencings of $\pi_2(S)$ with respect to~$p$.

Given $S \subseteq (\Z_{p} \rtimes_{\varphi} H) \setminus \{ (0,0) \}$ we first construct a quotient sequencing of $\pi_2(S)$ with respect to~$p$. We then use the polynomial method to show that there is a sequencing for $S$ that projects elementwise onto that quotient sequencing. Given such a quotient sequencing $\bm{a} = (a_1, \ldots a_k)$ with partial sums $\bm{b} = (b_0, b_1, \ldots, b_k)$, let
$$\bm{x_a} = \left(x_1\cdot_G a_1, x_2\cdot_G a_2, \ldots, x_k\cdot_G a_k \right)$$
be a putative arrangement of the elements of~$S$ with partial sums
$$\bm{y_a} = \left(s_0,\ldots,s_k\right)=\left(y_0\cdot_G b_0, y_1\cdot_G b_1, \ldots, y_k\cdot_G b_k \right).$$
We will need the following lemma:
\begin{lem}\label{SumOnN}
Let $(G,\cdot_G)$ be a group of type $\Z_{p} \rtimes_{\varphi} H$. Then, considering $x_1\cdot_G a_1$ and $x_2\cdot_G a_2$ in $G$, there exists $\varphi_{a_1}\in \Z_p$ such that:
$$(x_1\cdot_G a_1)\cdot_G (x_2\cdot_G a_2)=(x_1 + (\varphi_{a_1} x_2))\cdot_G (a_1\cdot_H a_2)$$
where $\varphi_{a_1}x_2$ represent the product between the elements of $\Z_p$, $\varphi_{a_1}$ and $x_2$.
\end{lem}
\begin{proof}
Due to the definition, we have that
$$(x_1\cdot_G a_1)\cdot_G (x_2\cdot_G a_2)=(x_1+ \varphi(a_1)(x_2))\cdot_G (a_1\cdot_H a_2).$$
Here the map
$$x\rightarrow \varphi(a_1)(x)$$
is an automorphism of $(\mathbb{Z}_p,+)$. It is well known that the automorphism of $(\mathbb{Z}_p,+)$ are the multiplications and hence there exists $\varphi_{a_1}\in \mathbb{Z}_p$ such that
$$\varphi(a_1)(x)=\varphi_{a_1}x.$$
Hence we have that
$$(x_1+ \varphi(a_1)(x_2))\cdot_G (a_1\cdot_H a_2)=(x_1 + (\varphi_{a_1} x_2))\cdot_G (a_1\cdot_H a_2)$$
which implies the thesis.
\end{proof}
From now on, when it is clear from the context, the index $G$ or $H$ in the products will be omitted.

Here we note that, if $j>i$, $s_i=s_j$ only when $$(x_{i+1}\cdot a_{i+1})\cdot (x_{i+2}\cdot a_{i+2})\cdot \ldots\cdot (x_{j}\cdot a_{j})=id.$$
Applying Lemma \ref{SumOnN}, this equality can be written as
$$(x_{i+1}+ \varphi_{a_{i+1}}x_{i+2})\cdot (a_{i+1}\cdot a_{i+2})\cdot \ldots\cdot (x_{j}\cdot a_{j})=id$$
and, reiterating it, as
$$(x_{i+1}+ \varphi_{a_{i+1}}x_{i+2}+\varphi_{a_{i+1}\cdot a_{i+2}}x_{i+3}+\dots+\varphi_{a_{i+1}\dots a_{j-1}}x_j)\cdot (a_{i+1}\cdot a_{i+2}\cdot\ldots\cdot a_{j})=id.$$
This implies that
$$a_{i+1}\cdot a_{i+2}\cdot \ldots\cdot a_{j} = id$$
which is $b_i=b_j$, and
$$x_{i+1}+ \varphi_{a_{i+1}}x_{i+2}+\varphi_{a_{i+1}\cdot a_{i+2}}x_{i+3}+\dots+\varphi_{a_{i+1}\dots a_{j-1}}x_j=0.$$
Hence, reasoning as in \cite{CDOR}, we can consider the following polynomial in the variables $x_1, x_2, \ldots, x_k$
\begin{equation}\label{pol1}p_{\bm{a}} = \prod_{\substack{1 \leq i < j \leq k \\ a_i = a_j }} (x_j - x_i)
\prod_{\substack{0 \leq i < j \leq k \\ b_i = b_j \\ j \neq i+1 }} (x_{i+1}+ \varphi_{a_{i+1}}x_{i+2}+\varphi_{a_{i+1}\cdot a_{i+2}}x_{i+3}+\dots+\varphi_{a_{i+1}\dots a_{j-1}}x_j) \end{equation}
where the variable~$x_i$ ranges over the values~$\{ c : c\cdot a_i \in S\}$ for each~$i$. Note that, if $G$ is non-abelian, there is no reason to exclude the pair $(0,k)$ in the second product of Equation \eqref{pol1}. Indeed we can even ask that $s_0\not=s_k$.
Now, it is clear that, given $S=\{x_1\cdot a_1, x_2\cdot a_2, \ldots, x_k\cdot a_k \}\subseteq G\setminus \{0\}$ of size $k$, $p_{\bm{a}} (x_1,\dots,x_k)\not=0$ only if the sequence $(x_1\cdot a_1, x_2\cdot a_2, \ldots, x_k\cdot a_k )$ is a sequencing for the set $S$.

\subsection{Computational Results}\label{subsec:comp1}
The method introduced in the previous paragraph generalizes that used in \cite{Ollis} to investigate the problem for dihedral groups. We recall the definition of these groups.
\begin{defn}\label{df:Dihedral}
The dihedral group $D_{2n}$ is the semidirect product $\Z_n\rtimes_{\varphi} \Z_2$ where
$$\varphi(0)=id \mbox{ and }\varphi(1)(x)=-x.$$
\end{defn}
In the following example, we see how the coefficients of Equation \eqref{pol1} behave for such groups.
\begin{ex}{Example}
Let $p$ be a prime,  $p > 3$,  and suppose $S \subseteq D_{2p} \setminus \{0\}$ with $|S|=5$ and of type~$(3,2)$.   The sequence $\bm{a} = (0,1,0,0,1)$ has partial sums $(0,  0,  1,  1,  1,  0)$ and so is a quotient sequencing of~$\pi_2(S)$ with respect to~$p$.
We desire a sequencing of~$S$ of the form 
$$\left( x_1 \cdot 0, x_2 \cdot 1,  x_3 \cdot 0,  x_4 \cdot 0,  x_5 \cdot 1 \right).$$
The polynomial is
\begin{equation*}
p_{\bm{a}} = (x_3 - x_1) (x_4 - x_1) (x_5 - x_2) (x_4 - x_3) (x_1 + x_2 - x_3 - x_4 - x_5) (x_2 - x_3 - x_4 - x_5) (x_3 + x_4) \,.
\end{equation*}
To apply the Non-Vanishing Corollary we need a monomial of this polynomial which divides the bounding monomial $x_1^2x_2 x_3^2 x_4^2 x_5$ with a nonzero coefficient.  One such is $x_1^2x_2 x_3^2 x_4 x_5^2$,  which has coefficient~$6$.   Hence whenever~$S$ has this form it has a sequencing. 
\end{ex}
Another class of groups that we will consider in this discussion is that of $G_{3p}$.
\begin{defn}\label{df:G3p}
Given $p\equiv 1 \pmod{6}$ and $r$ such that $r^3\equiv 1 \pmod{p}$ and $r\not\equiv 1 \pmod{p}$, $G_{3p}$ is defined as $\Z_p\rtimes_{\varphi} \Z_3$ where
$$\varphi(i)(x)=r^ix \mbox{ for any }i\in [0,1,2].$$
In case $r\equiv 1 \pmod{p}$, using the same construction we would obtain the direct product $\Z_p\times \Z_3$.
\end{defn}
In the following example, we see how the coefficients of Equation \eqref{pol1} behave for such groups.
\begin{ex}{Example}
Let $p$ be a prime,  $p > 3$,  and suppose $S \subseteq G_{3p} \setminus \{0\}$ with $|S|=5$ and of type~$(1,3,1)$.   The sequence $\bm{a} = (1,1,0,2,1)$ has partial sums $(0,  1,  2,  2,  1,  2)$ and so is a quotient sequencing of~$\pi_2(S)$ with respect to~$p$.
We desire a sequencing of~$S$ of the form 
$$\left( x_1 \cdot 1,  x_2 \cdot 1,  x_3 \cdot 0,  x_4 \cdot 2,   x_5 \cdot 1 \right).$$
The polynomial is
\begin{equation*}
p_{\bm{a}} =  (x_2 - x_1)(x_5 - x_1)(x_5 - x_2) (x_2 + r x_3 + r x_4) (x_3 + x_4 + r^2 x_5) (x_4 + r^2 x_5)\,.
\end{equation*}
To apply the Non-Vanishing Corollary we need a monomial of this polynomial which divides the bounding monomial $x_1^2x_2^2 x_5^2$ with a nonzero coefficient.  One such is $x_1^2x_2^2 x_5^2$,  which has coefficient~$-r^4 = -r$.   Hence whenever~$S$ has this form it has a sequencing.  
\end{ex}
We are now ready to prove the main result of this section.
\begin{thm}\label{th:main}
Let $n = pe$ with~$p$ prime and let $G$ be a group of size $n$. Then subsets~$S$ of size~$k$ of~$G$ are sequenceable in the following cases:
\begin{enumerate}
\item $k\not=4,5$ such that $k \leq 10$ and $e \in \{1,2,3\}$;
\item $k\leq 12$, $p>3$ and $e \in \{1,2\}$.
\end{enumerate}
\end{thm}
\begin{proof}
The case $e=1$ follows from the results of \cite{HOS19} and \cite{CDOR}. Since $\mathbb{Z}_p\rtimes_{\varphi} \{0\}$ is isomorphic to $\Z_p$, from the same papers, it follows the sequenceability of subsets $S$ of $D_{2p}$ (resp. $G_{3p}$) whose type is $(k,0)$ (resp. $(k,0,0)$).

Note that a group of size $2p$ is either $\Z_p\times \Z_2$ or the dihedral group $D_{2p}$. Analogously, a group of size $3p$ is either $\Z_p\times \Z_3$ or $G_{3p}$ for some $r$. Since the direct product case has already been considered in \cite{CDOR}, we consider here just the cases of $D_{2p}$ and $G_{3p}$. Also, for the dihedral groups, we can suppose that $k \geq 10$ since by \cite{Ollis} we know that subsets $S$ of size $k \leq 9$ in $D_{2p}$ (with $p>3$) are sequenceable.  We also know from \cite[Lemma 3.1]{Ollis} that subsets  $S \subseteq D_{2p}$ of size $k$ and of type $(0,k)$ are sequenceable for every $k$.
Then, in all of the other cases stated in the theorem, we can use the Non-Vanishing Corollary for suitable quotient sequences and nonzero monomial (in Tables \ref{tab:12_2} and \ref{tab:10_3} below we report the extreme cases). These results are obtained using the Python framework SageMath~\cite{SageMath}. 

Here, when considering $G_{3p}$, for each type we provide a set of coefficients whose expressions depend on $r$. We have checked that, for any prime $p\equiv 1 \pmod{6}$, at least one of these coefficients is nonzero proceeding as follows. The coefficients can be written in the form $ar+b$ for some $a,b$ whose expression do not depend on $p$. If $a\not\equiv 0\pmod{p}$, the coefficient is zero if and only if $r\equiv-b\cdot a^{-1}\pmod{p}$. Recalling that $r$ satisfies $r^2+r+1\equiv 0\pmod{p}$ we have that 
$a^{-2}(b^2-ab+a^2)\equiv 0\pmod{p}$
and hence, 
\begin{equation}\label{modp}
b^2-ab+a^2\equiv 0\pmod{p}.
\end{equation}
On the other hand, if $a\equiv 0 \pmod{p}$ the coefficient is zero only when also $b$ is zero modulo $p$ and hence Equation \eqref{modp} is satisfied also in this case. Since Equation \eqref{modp} is satisfied for finitely many $p$, one can check that, for any prime $p$, at least one of the coefficients presented in Table \ref{tab:10_3} is nonzero.
\end{proof}
\begin{rem}\label{rem:linear1}
Since in Equation \eqref{pol1} we have imposed also $s_0\not=s_k$, we have that in case $G=D_{2p}$ (resp. $G=G_{3p}$), $S$ is not of type $(k,0)$ (resp. $(k,0,0)$) and $k$ satisfies the hypotheses of the previous theorem, then $S$ admits a linear sequencing.

This implies that, under the hypothesis of Theorem \ref{th:main}, $S$ admits a linear sequencing unless $s_k$ is unavoidably the identity.
\end{rem}

The following Tables~\ref{tab:12_2} and \ref{tab:10_3} contain the required monomials and their coefficients for the proof of Theorem \ref{th:main} in the cases~$(G,k)=(D_{2p},12)$ and~$(G,k)=(G_{3p},10)$.

\footnotesize
\begin{longtable}{lllll}
\caption{Monomials and their coefficients for the proof of Theorem \ref{th:main} in the case $|S|=12$ and~$G=D_{2p}$.}\label{tab:12_2}\\
\hline
$\boldsymbol{\lambda}$ & $\mathbf{a}$ & deg & monomial/s & coefficient/s \\
\hline
\endfirsthead
\hline
$\boldsymbol{\lambda}$ & $\mathbf{a}$ & deg & monomial/s & coefficient/s \\
\hline
\endhead
$(11, 1)$ & \begin{tabular}{@{}l@{}} $(0, 0, 0, 0, 0, 1, 0, 0, 0, 0, 0, 0)$ \end{tabular} & \begin{tabular}{@{}l@{}} 71 \end{tabular} & \begin{tabular}{@{}l@{}} $x_{2}^{3}x_{3}^{5}x_{4}^{7}x_{5}^{9}x_{7}^{9}x_{8}^{9}x_{9}^{9}x_{10}^{9}x_{11}^{10}x_{12}^{10}$ \\ $x_{2}^{3}x_{3}^{5}x_{4}^{8}x_{5}^{8}x_{7}^{9}x_{8}^{9}x_{9}^{9}x_{10}^{9}x_{11}^{10}x_{12}^{10}$ \end{tabular} & \begin{tabular}{@{}l@{}} $2 \cdot 3 \cdot 7^2 \cdot 131$ \\ $2 \cdot 5 \cdot 7 \cdot 593$ \end{tabular}\\ \hline \\ 
$(10, 2)$ & \begin{tabular}{@{}l@{}} $(0, 0, 0, 0, 0, 1, 0, 0, 0, 0, 0, 1)$ \end{tabular} & \begin{tabular}{@{}l@{}} 64 \end{tabular} & \begin{tabular}{@{}l@{}} $x_{2}^{4}x_{3}^{5}x_{4}^{9}x_{5}^{8}x_{6}x_{7}^{8}x_{8}^{9}x_{9}^{9}x_{10}^{9}x_{11}^{9}x_{12}$ \\ $x_{2}^{4}x_{3}^{5}x_{4}^{8}x_{5}^{9}x_{6}x_{7}^{8}x_{8}^{9}x_{9}^{9}x_{10}^{9}x_{11}^{9}x_{12}$ \end{tabular} & \begin{tabular}{@{}l@{}} $- 2^7 \cdot 3 \cdot 29 \cdot 41$ \\ $- 2^5 \cdot 3 \cdot 101 \cdot 227$ \end{tabular}\\ \hline \\ 
$(9, 3)$ & \begin{tabular}{@{}l@{}} $(0, 0, 0, 0, 1, 1, 1, 0, 0, 0, 0, 0)$ \end{tabular} & \begin{tabular}{@{}l@{}} 57 \end{tabular} & \begin{tabular}{@{}l@{}} $x_{2}^{2}x_{3}^{8}x_{4}^{8}x_{5}^{2}x_{6}^{2}x_{7}^{2}x_{8}^{8}x_{9}^{8}x_{10}^{8}x_{11}^{8}x_{12}^{8}$ \\ $x_{2}^{3}x_{3}^{7}x_{4}^{8}x_{5}^{2}x_{6}^{2}x_{7}^{2}x_{8}^{8}x_{9}^{8}x_{10}^{8}x_{11}^{8}x_{12}^{8}$ \end{tabular} & \begin{tabular}{@{}l@{}} $- 2^3 \cdot 3^3 \cdot 5 \cdot 31$ \\ $- 2^2 \cdot 3^2 \cdot 61 \cdot 73$ \end{tabular}\\ \hline \\ 
$(8, 4)$ & \begin{tabular}{@{}l@{}} $(0, 0, 0, 0, 1, 1, 1, 0, 0, 0, 0, 1)$ \end{tabular} & \begin{tabular}{@{}l@{}} 54 \end{tabular} & \begin{tabular}{@{}l@{}} $x_{2}^{6}x_{3}^{7}x_{4}^{7}x_{5}^{3}x_{6}^{3}x_{7}^{3}x_{8}^{7}x_{9}^{7}x_{10}^{7}x_{11}^{7}x_{12}^{3}$ \\ $x_{2}^{7}x_{3}^{7}x_{4}^{7}x_{5}^{3}x_{6}^{3}x_{7}^{3}x_{8}^{6}x_{9}^{7}x_{10}^{7}x_{11}^{7}x_{12}^{3}$ \end{tabular} & \begin{tabular}{@{}l@{}} $2^6 \cdot 7 \cdot 7669$ \\ $- 2^6 \cdot 11 \cdot 28591$ \end{tabular}\\ \hline \\ 
$(7, 5)$ & \begin{tabular}{@{}l@{}} $(0, 0, 0, 1, 1, 1, 1, 1, 0, 0, 0, 0)$ \end{tabular} & \begin{tabular}{@{}l@{}} 51 \end{tabular} & \begin{tabular}{@{}l@{}} $x_{2}^{6}x_{3}^{6}x_{4}^{4}x_{5}^{4}x_{6}^{4}x_{7}^{4}x_{8}^{4}x_{9}^{6}x_{10}^{6}x_{11}^{6}x_{12}^{6}$ \\ $x_{1}^{6}x_{3}^{6}x_{4}^{4}x_{5}^{4}x_{6}^{4}x_{7}^{4}x_{8}^{4}x_{9}^{6}x_{10}^{6}x_{11}^{6}x_{12}^{6}$ \end{tabular} & \begin{tabular}{@{}l@{}} $- 2^3 \cdot 3^2 \cdot 5 \cdot 7 \cdot 11 \cdot 401$ \\ $- 2^2 \cdot 5^2 \cdot 179 \cdot 389$ \end{tabular}\\ \hline \\ 
$(6, 6)$ & \begin{tabular}{@{}l@{}} $(0, 0, 0, 1, 1, 1, 1, 1, 0, 0, 0, 1)$ \end{tabular} & \begin{tabular}{@{}l@{}} 52 \end{tabular} & \begin{tabular}{@{}l@{}} $x_{1}^{2}x_{2}^{4}x_{3}^{5}x_{4}^{5}x_{5}^{5}x_{6}^{5}x_{7}^{5}x_{8}^{5}x_{9}^{5}x_{10}^{5}x_{11}^{5}x_{12}^{5}$ \\ $x_{1}^{3}x_{2}^{3}x_{3}^{5}x_{4}^{5}x_{5}^{5}x_{6}^{5}x_{7}^{5}x_{8}^{5}x_{9}^{5}x_{10}^{5}x_{11}^{5}x_{12}^{5}$ \end{tabular} & \begin{tabular}{@{}l@{}} $2 \cdot 3^3 \cdot 181 \cdot 481373$ \\ $- 2 \cdot 3^2 \cdot 13 \cdot 22399121$ \end{tabular}\\ \hline \\ 
$(4, 8)$ & \begin{tabular}{@{}l@{}} $(0, 0, 1, 1, 1, 1, 1, 1, 1, 0, 0, 1)$ \end{tabular} & \begin{tabular}{@{}l@{}} 58 \end{tabular} & \begin{tabular}{@{}l@{}} $x_{2}^{3}x_{3}^{4}x_{4}^{5}x_{5}^{7}x_{6}^{7}x_{7}^{7}x_{8}^{7}x_{9}^{7}x_{10}^{3}x_{11}^{3}x_{12}^{7}$ \\ $x_{2}^{3}x_{3}^{4}x_{4}^{6}x_{5}^{6}x_{6}^{7}x_{7}^{7}x_{8}^{7}x_{9}^{7}x_{10}^{3}x_{11}^{3}x_{12}^{7}$ \end{tabular} & \begin{tabular}{@{}l@{}} $2^4 \cdot 3 \cdot 79 \cdot 179 \cdot 18979$ \\ $- 2^4 \cdot 3^2 \cdot 5 \cdot 109 \cdot 119839$ \end{tabular}\\ \hline \\ 
$(3, 9)$ & \begin{tabular}{@{}l@{}} $(0, 1, 1, 1, 1, 1, 1, 1, 1, 1, 0, 0)$ \end{tabular} & \begin{tabular}{@{}l@{}} 63 \end{tabular} & \begin{tabular}{@{}l@{}} $x_{1}^{2}x_{2}^{4}x_{3}^{6}x_{4}^{8}x_{5}^{8}x_{7}^{8}x_{8}^{8}x_{9}^{8}x_{10}^{8}x_{11}^{2}x_{12}^{2}$ \\ $x_{1}^{2}x_{2}^{5}x_{3}^{5}x_{4}^{8}x_{5}^{8}x_{7}^{8}x_{8}^{8}x_{9}^{8}x_{10}^{8}x_{11}^{2}x_{12}^{2}$ \end{tabular} & \begin{tabular}{@{}l@{}} $3 \cdot 62701 \cdot 131111$ \\ $- 2 \cdot 3 \cdot 5 \cdot 13 \cdot 17 \cdot 3699679$ \end{tabular}\\ \hline \\ 
$(2, 10)$ & \begin{tabular}{@{}l@{}} $(0, 1, 1, 1, 1, 1, 1, 1, 1, 1, 0, 1)$ \end{tabular} & \begin{tabular}{@{}l@{}} 72 \end{tabular} & \begin{tabular}{@{}l@{}} $x_{1}x_{2}^{4}x_{3}^{4}x_{4}^{8}x_{5}^{9}x_{7}^{9}x_{8}^{9}x_{9}^{9}x_{10}^{9}x_{11}x_{12}^{9}$ \\ $x_{1}x_{2}^{4}x_{3}^{5}x_{4}^{7}x_{5}^{9}x_{7}^{9}x_{8}^{9}x_{9}^{9}x_{10}^{9}x_{11}x_{12}^{9}$ \end{tabular} & \begin{tabular}{@{}l@{}} $2 \cdot 5 \cdot 71 \cdot 1427 \cdot 1033463$ \\ $- 2 \cdot 3^2 \cdot 7 \cdot 23 \cdot 215947813$ \end{tabular}\\ \hline \\ 
$(1, 11)$ & \begin{tabular}{@{}l@{}} $(1, 1, 1, 1, 1, 1, 1, 1, 1, 1, 1, 0)$ \end{tabular} & \begin{tabular}{@{}l@{}} 80 \end{tabular} & \begin{tabular}{@{}l@{}} $x_{2}^{5}x_{3}^{5}x_{4}^{7}x_{5}^{7}x_{6}^{8}x_{7}^{8}x_{8}^{10}x_{9}^{10}x_{10}^{10}x_{11}^{10}$ \\ $x_{2}^{4}x_{3}^{6}x_{4}^{7}x_{5}^{7}x_{6}^{8}x_{7}^{8}x_{8}^{10}x_{9}^{10}x_{10}^{10}x_{11}^{10}$ \end{tabular} & \begin{tabular}{@{}l@{}} $2^3 \cdot 7 \cdot 101 \cdot 81083329$ \\ $- 2^5 \cdot 5 \cdot 31 \cdot 3557 \cdot 21647$ \end{tabular}
\\ \hline
%
\end{longtable}
\footnotesize
\begin{longtable}{lllll}
\caption{Monomials and their coefficients for the proof of Theorem \ref{th:main} in the case $|S|=10$ and~$G=G_{3p}$.}\label{tab:10_3}\\
\hline
$\boldsymbol{\lambda}$ & $\mathbf{a}$ & deg & monomial/s & coefficient/s \\
\hline
\endfirsthead
\hline
$\boldsymbol{\lambda}$ & $\mathbf{a}$ & deg & monomial/s & coefficient/s \\
\hline
\endhead
$(9, 1, 0)$ & \begin{tabular}{@{}l@{}} $(0, 0, 0, 0, 0, 1, 0, 0, 0, 0)$ \end{tabular} & \begin{tabular}{@{}l@{}} 52 \end{tabular} & \begin{tabular}{@{}l@{}} $x_{2}^{2}x_{3}^{4}x_{4}^{7}x_{5}^{8}x_{7}^{7}x_{8}^{8}x_{9}^{8}x_{10}^{8}$ \end{tabular} & \begin{tabular}{@{}l@{}} $-2^2$ \end{tabular}\\ \hline \\ 
$(8, 2, 0)$ & \begin{tabular}{@{}l@{}} $(0, 0, 0, 1, 0, 0, 0, 1, 0, 0)$ \end{tabular} & \begin{tabular}{@{}l@{}} 36 \end{tabular} & \begin{tabular}{@{}l@{}} $x_{2}^{2}x_{3}^{4}x_{5}^{3}x_{6}^{6}x_{7}^{7}x_{8}x_{9}^{6}x_{10}^{7}$ \end{tabular} & \begin{tabular}{@{}l@{}} $-2$ \end{tabular}\\ \hline \\ 
$(8, 1, 1)$ & \begin{tabular}{@{}l@{}} $(0, 0, 0, 0, 1, 0, 0, 0, 0, 2)$ \end{tabular} & \begin{tabular}{@{}l@{}} 45 \end{tabular} & \begin{tabular}{@{}l@{}} $x_{1}x_{2}^{4}x_{3}^{6}x_{4}^{7}x_{6}^{6}x_{7}^{7}x_{8}^{7}x_{9}^{7}$ \\ $x_{1}^{2}x_{2}^{3}x_{3}^{6}x_{4}^{7}x_{6}^{6}x_{7}^{7}x_{8}^{7}x_{9}^{7}$ \end{tabular} & \begin{tabular}{@{}l@{}} $-1791r + 573$ \\ $-235r - 746$ \end{tabular}\\ \hline \\ 
$(7, 3, 0)$ & \begin{tabular}{@{}l@{}} $(0, 0, 1, 0, 0, 0, 1, 0, 0, 1)$ \end{tabular} & \begin{tabular}{@{}l@{}} 32 \end{tabular} & \begin{tabular}{@{}l@{}} $x_{1}^{2}x_{2}^{3}x_{4}^{2}x_{5}^{5}x_{6}^{6}x_{7}x_{8}^{5}x_{9}^{6}x_{10}^{2}$ \\ $x_{1}^{2}x_{2}^{2}x_{4}^{3}x_{5}^{5}x_{6}^{6}x_{7}x_{8}^{5}x_{9}^{6}x_{10}^{2}$ \end{tabular} & \begin{tabular}{@{}l@{}} $14r + 73$ \\ $12r + 21$ \end{tabular}\\ \hline \\ 
$(7, 2, 1)$ & \begin{tabular}{@{}l@{}} $(0, 0, 0, 1, 0, 0, 1, 0, 0, 2)$ \end{tabular} & \begin{tabular}{@{}l@{}} 30 \end{tabular} & \begin{tabular}{@{}l@{}} $x_{1}x_{2}^{2}x_{3}^{4}x_{4}x_{5}^{5}x_{6}^{5}x_{7}x_{8}^{5}x_{9}^{6}$ \end{tabular} & \begin{tabular}{@{}l@{}} $1$ \end{tabular}\\ \hline \\ 
$(6, 4, 0)$ & \begin{tabular}{@{}l@{}} $(0, 0, 1, 0, 0, 1, 0, 0, 1, 1)$ \end{tabular} & \begin{tabular}{@{}l@{}} 30 \end{tabular} & \begin{tabular}{@{}l@{}} $x_{2}^{2}x_{3}^{2}x_{4}^{4}x_{5}^{5}x_{6}^{2}x_{7}^{4}x_{8}^{5}x_{9}^{3}x_{10}^{3}$ \\ $x_{1}x_{2}^{2}x_{3}^{2}x_{4}^{3}x_{5}^{5}x_{6}^{2}x_{7}^{4}x_{8}^{5}x_{9}^{3}x_{10}^{3}$ \end{tabular} & \begin{tabular}{@{}l@{}} $-1168r + 72$ \\ $-436r - 1078$ \end{tabular}\\ \hline \\ 
$(6, 3, 1)$ & \begin{tabular}{@{}l@{}} $(0, 0, 0, 1, 0, 0, 1, 0, 2, 1)$ \end{tabular} & \begin{tabular}{@{}l@{}} 27 \end{tabular} & \begin{tabular}{@{}l@{}} $x_{2}^{2}x_{3}^{5}x_{4}x_{5}^{5}x_{6}^{5}x_{7}^{2}x_{8}^{5}x_{10}^{2}$ \end{tabular} & \begin{tabular}{@{}l@{}} $r + 1$ \end{tabular}\\ \hline \\ 
$(6, 2, 2)$ & \begin{tabular}{@{}l@{}} $(0, 0, 1, 0, 0, 1, 0, 0, 2, 2)$ \end{tabular} & \begin{tabular}{@{}l@{}} 26 \end{tabular} & \begin{tabular}{@{}l@{}} $x_{1}^{3}x_{2}^{3}x_{3}x_{4}^{3}x_{5}^{4}x_{6}x_{7}^{4}x_{8}^{5}x_{9}x_{10}$ \\ $x_{1}^{2}x_{2}^{5}x_{4}^{3}x_{5}^{4}x_{6}x_{7}^{4}x_{8}^{5}x_{9}x_{10}$ \end{tabular} & \begin{tabular}{@{}l@{}} $10r - 64$ \\ $155r + 583$ \end{tabular}\\ \hline \\ 
$(5, 5, 0)$ & \begin{tabular}{@{}l@{}} $(0, 0, 1, 0, 0, 1, 0, 1, 1, 1)$ \end{tabular} & \begin{tabular}{@{}l@{}} 30 \end{tabular} & \begin{tabular}{@{}l@{}} $x_{2}^{2}x_{3}x_{4}^{3}x_{5}^{4}x_{6}^{4}x_{7}^{4}x_{8}^{4}x_{9}^{4}x_{10}^{4}$ \\ $x_{1}x_{2}^{2}x_{3}^{2}x_{4}^{2}x_{5}^{3}x_{6}^{4}x_{7}^{4}x_{8}^{4}x_{9}^{4}x_{10}^{4}$ \end{tabular} & \begin{tabular}{@{}l@{}} $1965r + 2715$ \\ $385r - 190$ \end{tabular}\\ \hline \\ 
$(5, 4, 1)$ & \begin{tabular}{@{}l@{}} $(0, 0, 1, 0, 0, 1, 0, 2, 1, 1)$ \end{tabular} & \begin{tabular}{@{}l@{}} 26 \end{tabular} & \begin{tabular}{@{}l@{}} $x_{1}x_{2}^{4}x_{3}x_{4}^{3}x_{5}^{4}x_{6}^{3}x_{7}^{4}x_{9}^{3}x_{10}^{3}$ \\ $x_{1}^{2}x_{2}^{3}x_{3}x_{4}^{3}x_{5}^{4}x_{6}^{3}x_{7}^{4}x_{9}^{3}x_{10}^{3}$ \end{tabular} & \begin{tabular}{@{}l@{}} $30r + 156$ \\ $42r + 72$ \end{tabular}\\ \hline \\ 
$(5, 3, 2)$ & \begin{tabular}{@{}l@{}} $(0, 0, 0, 1, 0, 1, 0, 2, 1, 2)$ \end{tabular} & \begin{tabular}{@{}l@{}} 24 \end{tabular} & \begin{tabular}{@{}l@{}} $x_{1}^{4}x_{2}^{4}x_{3}^{3}x_{4}^{2}x_{5}^{4}x_{6}^{2}x_{7}^{2}x_{8}x_{9}^{2}$ \end{tabular} & \begin{tabular}{@{}l@{}} $2$ \end{tabular}\\ \hline \\ 
$(4, 6, 0)$ & \begin{tabular}{@{}l@{}} $(0, 1, 0, 0, 1, 0, 1, 1, 1, 1)$ \end{tabular} & \begin{tabular}{@{}l@{}} 32 \end{tabular} & \begin{tabular}{@{}l@{}} $x_{1}^{3}x_{2}^{5}x_{3}^{3}x_{4}^{3}x_{5}^{5}x_{6}^{3}x_{7}^{4}x_{8}^{4}x_{9}x_{10}$ \\ $x_{1}^{2}x_{2}^{5}x_{3}^{3}x_{4}^{3}x_{5}^{5}x_{6}^{3}x_{7}^{5}x_{8}^{4}x_{9}x_{10}$ \end{tabular} & \begin{tabular}{@{}l@{}} $-11377r + 11374$ \\ $-21587r - 27358$ \end{tabular}\\ \hline \\ 
$(4, 5, 1)$ & \begin{tabular}{@{}l@{}} $(0, 0, 1, 0, 1, 0, 1, 1, 1, 2)$ \end{tabular} & \begin{tabular}{@{}l@{}} 27 \end{tabular} & \begin{tabular}{@{}l@{}} $x_{2}^{2}x_{3}^{3}x_{4}^{3}x_{5}^{4}x_{6}^{3}x_{7}^{4}x_{8}^{4}x_{9}^{4}$ \\ $x_{1}^{2}x_{2}^{2}x_{3}^{3}x_{4}^{3}x_{5}^{3}x_{6}^{3}x_{7}^{3}x_{8}^{4}x_{9}^{4}$ \end{tabular} & \begin{tabular}{@{}l@{}} $-1286r + 68$ \\ $-333r - 119$ \end{tabular}\\ \hline \\ 
$(4, 4, 2)$ & \begin{tabular}{@{}l@{}} $(0, 0, 0, 1, 0, 1, 2, 1, 2, 1)$ \end{tabular} & \begin{tabular}{@{}l@{}} 24 \end{tabular} & \begin{tabular}{@{}l@{}} $x_{1}^{2}x_{2}^{2}x_{3}^{3}x_{4}^{3}x_{5}^{3}x_{6}^{3}x_{7}x_{8}^{3}x_{9}x_{10}^{3}$ \end{tabular} & \begin{tabular}{@{}l@{}} $-2r$ \end{tabular}\\ \hline \\ 
$(4, 3, 3)$ & \begin{tabular}{@{}l@{}} $(0, 0, 1, 0, 1, 0, 2, 1, 2, 2)$ \end{tabular} & \begin{tabular}{@{}l@{}} 23 \end{tabular} & \begin{tabular}{@{}l@{}} $x_{1}^{2}x_{2}^{3}x_{3}^{2}x_{4}^{3}x_{5}^{2}x_{6}^{3}x_{7}^{2}x_{8}^{2}x_{9}^{2}x_{10}^{2}$ \\ $x_{1}^{3}x_{2}^{3}x_{3}^{2}x_{4}^{3}x_{5}^{2}x_{6}^{2}x_{7}^{2}x_{8}^{2}x_{9}^{2}x_{10}^{2}$ \end{tabular} & \begin{tabular}{@{}l@{}} $-1452r + 1686$ \\ $-759r - 348$ \end{tabular}\\ \hline \\ 
$(3, 7, 0)$ & \begin{tabular}{@{}l@{}} $(0, 1, 0, 1, 0, 1, 1, 1, 1, 1)$ \end{tabular} & \begin{tabular}{@{}l@{}} 36 \end{tabular} & \begin{tabular}{@{}l@{}} $x_{2}x_{3}^{2}x_{4}x_{5}^{2}x_{6}^{6}x_{7}^{6}x_{8}^{6}x_{9}^{6}x_{10}^{6}$ \\ $x_{2}x_{3}^{2}x_{4}^{3}x_{5}^{2}x_{6}^{4}x_{7}^{6}x_{8}^{6}x_{9}^{6}x_{10}^{6}$ \end{tabular} & \begin{tabular}{@{}l@{}} $-19616r + 331264$ \\ $802469r - 381850$ \end{tabular}\\ \hline \\ 
$(3, 6, 1)$ & \begin{tabular}{@{}l@{}} $(0, 0, 1, 0, 1, 1, 1, 1, 2, 1)$ \end{tabular} & \begin{tabular}{@{}l@{}} 30 \end{tabular} & \begin{tabular}{@{}l@{}} $x_{1}^{2}x_{2}x_{3}^{5}x_{4}^{2}x_{5}^{5}x_{6}^{5}x_{7}^{5}x_{8}^{4}x_{10}$ \\ $x_{1}^{2}x_{2}x_{3}^{3}x_{4}^{2}x_{5}^{5}x_{6}^{5}x_{7}^{5}x_{8}^{4}x_{10}^{3}$ \end{tabular} & \begin{tabular}{@{}l@{}} $3396r - 230$ \\ $133r + 344$ \end{tabular}\\ \hline \\ 
$(3, 5, 2)$ & \begin{tabular}{@{}l@{}} $(0, 0, 1, 0, 1, 2, 1, 2, 1, 1)$ \end{tabular} & \begin{tabular}{@{}l@{}} 26 \end{tabular} & \begin{tabular}{@{}l@{}} $x_{1}^{2}x_{2}^{2}x_{3}^{2}x_{4}^{2}x_{5}^{4}x_{6}x_{7}^{4}x_{8}x_{9}^{4}x_{10}^{4}$ \\ $x_{1}^{2}x_{2}^{2}x_{3}^{3}x_{4}^{2}x_{5}^{3}x_{6}x_{7}^{4}x_{8}x_{9}^{4}x_{10}^{4}$ \end{tabular} & \begin{tabular}{@{}l@{}} $-3r + 24$ \\ $21r + 21$ \end{tabular}\\ \hline \\ 
$(3, 4, 3)$ & \begin{tabular}{@{}l@{}} $(0, 0, 1, 0, 1, 1, 2, 2, 1, 2)$ \\ $(0, 0, 1, 1, 0, 1, 2, 2, 1, 2)$ \end{tabular} & \begin{tabular}{@{}l@{}} 24 \\ 24 \end{tabular} & \begin{tabular}{@{}l@{}} $x_{1}^{2}x_{2}^{2}x_{3}^{3}x_{4}^{2}x_{5}^{3}x_{6}^{3}x_{7}^{2}x_{8}^{2}x_{9}^{3}x_{10}^{2}$ \\ $x_{1}^{2}x_{2}^{2}x_{3}^{3}x_{4}^{3}x_{5}^{2}x_{6}^{3}x_{7}^{2}x_{8}^{2}x_{9}^{3}x_{10}^{2}$ \end{tabular} & \begin{tabular}{@{}l@{}} $1358r + 1812$ \\ $-1742r - 1035$ \end{tabular}\\ \hline \\ 
$(2, 8, 0)$ & \begin{tabular}{@{}l@{}} $(0, 1, 0, 1, 1, 1, 1, 1, 1, 1)$ \end{tabular} & \begin{tabular}{@{}l@{}} 42 \end{tabular} & \begin{tabular}{@{}l@{}} $x_{3}x_{4}^{2}x_{5}^{4}x_{6}^{7}x_{7}^{7}x_{8}^{7}x_{9}^{7}x_{10}^{7}$ \\ $x_{3}x_{4}^{2}x_{5}^{5}x_{6}^{6}x_{7}^{7}x_{8}^{7}x_{9}^{7}x_{10}^{7}$ \end{tabular} & \begin{tabular}{@{}l@{}} $65374r + 67999$ \\ $-160333r - 7057$ \end{tabular}\\ \hline \\ 
$(2, 7, 1)$ & \begin{tabular}{@{}l@{}} $(0, 1, 0, 1, 1, 1, 1, 2, 1, 1)$ \end{tabular} & \begin{tabular}{@{}l@{}} 35 \end{tabular} & \begin{tabular}{@{}l@{}} $x_{1}x_{3}x_{4}^{3}x_{5}^{6}x_{6}^{6}x_{7}^{6}x_{9}^{6}x_{10}^{6}$ \\ $x_{1}x_{3}x_{4}^{4}x_{5}^{5}x_{6}^{6}x_{7}^{6}x_{9}^{6}x_{10}^{6}$ \end{tabular} & \begin{tabular}{@{}l@{}} $17861r - 27263$ \\ $12821r + 26143$ \end{tabular}\\ \hline \\ 
$(2, 6, 2)$ & \begin{tabular}{@{}l@{}} $(2, 0, 1, 1, 1, 1, 1, 2, 1, 0)$ \end{tabular} & \begin{tabular}{@{}l@{}} 30 \end{tabular} & \begin{tabular}{@{}l@{}} $x_{3}^{3}x_{4}^{5}x_{5}^{5}x_{6}^{5}x_{7}^{5}x_{8}x_{9}^{5}x_{10}$ \\ $x_{3}^{4}x_{4}^{4}x_{5}^{5}x_{6}^{5}x_{7}^{5}x_{8}x_{9}^{5}x_{10}$ \end{tabular} & \begin{tabular}{@{}l@{}} $-292r + 246$ \\ $85r + 126$ \end{tabular}\\ \hline \\ 
$(2, 5, 3)$ & \begin{tabular}{@{}l@{}} $(2, 1, 1, 1, 0, 2, 2, 1, 0, 1)$ \end{tabular} & \begin{tabular}{@{}l@{}} 27 \end{tabular} & \begin{tabular}{@{}l@{}} $x_{1}^{2}x_{2}^{4}x_{3}^{4}x_{4}^{4}x_{6}^{2}x_{7}^{2}x_{8}^{4}x_{9}x_{10}^{4}$ \\ $x_{1}^{2}x_{2}^{3}x_{3}^{4}x_{4}^{4}x_{5}x_{6}^{2}x_{7}^{2}x_{8}^{4}x_{9}x_{10}^{4}$ \end{tabular} & \begin{tabular}{@{}l@{}} $17651r - 23546$ \\ $-10860r - 152465$ \end{tabular}\\ \hline \\ 
$(2, 4, 4)$ & \begin{tabular}{@{}l@{}} $(0, 1, 0, 1, 1, 2, 2, 1, 2, 2)$ \\ $(0, 1, 1, 0, 1, 2, 2, 1, 2, 2)$ \end{tabular} & \begin{tabular}{@{}l@{}} 26 \\ 26 \end{tabular} & \begin{tabular}{@{}l@{}} $x_{1}x_{2}^{3}x_{3}x_{4}^{3}x_{5}^{3}x_{6}^{3}x_{7}^{3}x_{8}^{3}x_{9}^{3}x_{10}^{3}$ \\ $x_{1}x_{2}^{3}x_{3}^{3}x_{4}x_{5}^{3}x_{6}^{3}x_{7}^{3}x_{8}^{3}x_{9}^{3}x_{10}^{3}$ \end{tabular} & \begin{tabular}{@{}l@{}} $136392r + 247068$ \\ $-175716r - 167652$ \end{tabular}\\ \hline \\ 
$(1, 9, 0)$ & \begin{tabular}{@{}l@{}} $(1, 0, 1, 1, 1, 1, 1, 1, 1, 1)$ \end{tabular} & \begin{tabular}{@{}l@{}} 50 \end{tabular} & \begin{tabular}{@{}l@{}} $x_{3}^{2}x_{4}^{4}x_{5}^{5}x_{6}^{7}x_{7}^{8}x_{8}^{8}x_{9}^{8}x_{10}^{8}$ \\ $x_{3}^{3}x_{4}^{3}x_{5}^{5}x_{6}^{7}x_{7}^{8}x_{8}^{8}x_{9}^{8}x_{10}^{8}$ \end{tabular} & \begin{tabular}{@{}l@{}} $992906r + 1427608$ \\ $-1815167r - 733348$ \end{tabular}\\ \hline \\ 
$(1, 8, 1)$ & \begin{tabular}{@{}l@{}} $(0, 1, 1, 1, 1, 1, 1, 1, 1, 2)$ \end{tabular} & \begin{tabular}{@{}l@{}} 42 \end{tabular} & \begin{tabular}{@{}l@{}} $x_{2}x_{3}^{2}x_{4}^{4}x_{5}^{7}x_{6}^{7}x_{7}^{7}x_{8}^{7}x_{9}^{7}$ \\ $x_{2}x_{3}^{2}x_{4}^{5}x_{5}^{6}x_{6}^{7}x_{7}^{7}x_{8}^{7}x_{9}^{7}$ \end{tabular} & \begin{tabular}{@{}l@{}} $63854r - 6907$ \\ $38518r + 68818$ \end{tabular}\\ \hline \\ 
$(1, 7, 2)$ & \begin{tabular}{@{}l@{}} $(0, 1, 1, 1, 1, 1, 1, 2, 2, 1)$ \end{tabular} & \begin{tabular}{@{}l@{}} 36 \end{tabular} & \begin{tabular}{@{}l@{}} $x_{3}^{4}x_{4}^{6}x_{5}^{6}x_{6}^{6}x_{7}^{6}x_{8}x_{9}x_{10}^{6}$ \\ $x_{3}^{5}x_{4}^{5}x_{5}^{6}x_{6}^{6}x_{7}^{6}x_{8}x_{9}x_{10}^{6}$ \end{tabular} & \begin{tabular}{@{}l@{}} $-16124r - 27940$ \\ $3718r + 13449$ \end{tabular}\\ \hline \\ 
$(1, 6, 3)$ & \begin{tabular}{@{}l@{}} $(1, 1, 0, 1, 1, 1, 2, 1, 2, 2)$ \end{tabular} & \begin{tabular}{@{}l@{}} 32 \end{tabular} & \begin{tabular}{@{}l@{}} $x_{1}x_{2}^{5}x_{4}^{5}x_{5}^{5}x_{6}^{5}x_{7}^{2}x_{8}^{5}x_{9}^{2}x_{10}^{2}$ \\ $x_{1}^{2}x_{2}^{4}x_{4}^{5}x_{5}^{5}x_{6}^{5}x_{7}^{2}x_{8}^{5}x_{9}^{2}x_{10}^{2}$ \end{tabular} & \begin{tabular}{@{}l@{}} $31065r + 30039$ \\ $1764r + 9147$ \end{tabular}\\ \hline \\ 
$(1, 5, 4)$ & \begin{tabular}{@{}l@{}} $(0, 1, 1, 1, 1, 1, 2, 2, 2, 2)$ \end{tabular} & \begin{tabular}{@{}l@{}} 30 \end{tabular} & \begin{tabular}{@{}l@{}} $x_{2}^{4}x_{3}^{4}x_{4}^{4}x_{5}^{4}x_{6}^{4}x_{7}^{3}x_{8}^{3}x_{9}^{3}x_{10}$ \\ $x_{2}^{2}x_{3}^{4}x_{4}^{4}x_{5}^{4}x_{6}^{4}x_{7}^{3}x_{8}^{3}x_{9}^{3}x_{10}^{3}$ \end{tabular} & \begin{tabular}{@{}l@{}} $211692r + 376028$ \\ $78700r - 731650$ \end{tabular}\\ \hline \\ 
$(0, 10, 0)$ & \begin{tabular}{@{}l@{}} $(1, 1, 1, 1, 1, 1, 1, 1, 1, 1)$ \end{tabular} & \begin{tabular}{@{}l@{}} 60 \end{tabular} & \begin{tabular}{@{}l@{}} $x_{1}^{2}x_{2}^{2}x_{3}^{4}x_{4}^{7}x_{5}^{9}x_{7}^{9}x_{8}^{9}x_{9}^{9}x_{10}^{9}$ \\ $x_{1}^{2}x_{2}^{2}x_{3}^{5}x_{4}^{6}x_{5}^{9}x_{7}^{9}x_{8}^{9}x_{9}^{9}x_{10}^{9}$ \end{tabular} & \begin{tabular}{@{}l@{}} $-250444r - 433547$ \\ $1040607r + 497898$ \end{tabular}\\ \hline \\ 
$(0, 9, 1)$ & \begin{tabular}{@{}l@{}} $(1, 1, 1, 1, 1, 2, 1, 1, 1, 1)$ \end{tabular} & \begin{tabular}{@{}l@{}} 51 \end{tabular} & \begin{tabular}{@{}l@{}} $x_{2}^{2}x_{3}^{3}x_{4}^{6}x_{5}^{8}x_{7}^{8}x_{8}^{8}x_{9}^{8}x_{10}^{8}$ \\ $x_{2}^{2}x_{3}^{4}x_{4}^{5}x_{5}^{8}x_{7}^{8}x_{8}^{8}x_{9}^{8}x_{10}^{8}$ \end{tabular} & \begin{tabular}{@{}l@{}} $21957r - 14546$ \\ $37855r + 47642$ \end{tabular}\\ \hline \\ 
$(0, 8, 2)$ & \begin{tabular}{@{}l@{}} $(1, 1, 1, 1, 1, 1, 1, 1, 2, 2)$ \end{tabular} & \begin{tabular}{@{}l@{}} 44 \end{tabular} & \begin{tabular}{@{}l@{}} $x_{2}^{3}x_{3}^{4}x_{4}^{7}x_{5}^{7}x_{6}^{7}x_{7}^{7}x_{8}^{7}x_{9}x_{10}$ \\ $x_{2}^{3}x_{3}^{5}x_{4}^{6}x_{5}^{7}x_{6}^{7}x_{7}^{7}x_{8}^{7}x_{9}x_{10}$ \end{tabular} & \begin{tabular}{@{}l@{}} $1698588r + 406932$ \\ $201528r + 1213386$ \end{tabular}\\ \hline \\ 
$(0, 7, 3)$ & \begin{tabular}{@{}l@{}} $(2, 1, 1, 1, 1, 1, 1, 1, 2, 2)$ \end{tabular} & \begin{tabular}{@{}l@{}} 39 \end{tabular} & \begin{tabular}{@{}l@{}} $x_{1}x_{2}^{2}x_{3}^{2}x_{4}^{6}x_{5}^{6}x_{6}^{6}x_{7}^{6}x_{8}^{6}x_{9}^{2}x_{10}^{2}$ \\ $x_{1}x_{2}^{2}x_{3}^{3}x_{4}^{5}x_{5}^{6}x_{6}^{6}x_{7}^{6}x_{8}^{6}x_{9}^{2}x_{10}^{2}$ \end{tabular} & \begin{tabular}{@{}l@{}} $824103r + 258808$ \\ $89629r + 972050$ \end{tabular}\\ \hline \\ 
$(0, 6, 4)$ & \begin{tabular}{@{}l@{}} $(2, 1, 2, 1, 1, 1, 1, 1, 2, 2)$ \end{tabular} & \begin{tabular}{@{}l@{}} 38 \end{tabular} & \begin{tabular}{@{}l@{}} $x_{1}x_{2}^{3}x_{3}^{3}x_{4}^{5}x_{5}^{5}x_{6}^{5}x_{7}^{5}x_{8}^{5}x_{9}^{3}x_{10}^{3}$ \\ $x_{2}^{4}x_{3}^{3}x_{4}^{5}x_{5}^{5}x_{6}^{5}x_{7}^{5}x_{8}^{5}x_{9}^{3}x_{10}^{3}$ \end{tabular} & \begin{tabular}{@{}l@{}} $1390432r + 1536910$ \\ $-792953r - 108625$ \end{tabular}\\ \hline \\ 
$(0, 5, 5)$ & \begin{tabular}{@{}l@{}} $(1, 1, 1, 1, 1, 2, 2, 2, 2, 2)$ \end{tabular} & \begin{tabular}{@{}l@{}} 35 \end{tabular} & \begin{tabular}{@{}l@{}} $x_{1}x_{2}^{2}x_{3}^{4}x_{4}^{4}x_{5}^{4}x_{6}^{4}x_{7}^{4}x_{8}^{4}x_{9}^{4}x_{10}^{4}$ \\ $x_{1}^{2}x_{2}^{2}x_{3}^{3}x_{4}^{4}x_{5}^{4}x_{6}^{4}x_{7}^{4}x_{8}^{4}x_{9}^{4}x_{10}^{4}$ \end{tabular} & \begin{tabular}{@{}l@{}} $-710283r + 3022047$ \\ $2272296r + 657732$ \end{tabular}\\ \hline
\end{longtable}

\normalsize

\section{To Infinity and Back}\label{sec:infinity}
In this section we apply the method of \cite{CP20} (see also \cite{CDOR}) to generalize the sequenceability results here obtained to a more general class of semidirect products.

First, given a subset $S$ of a group $G$, we define the set
$$\Upsilon(S)=S\cup\Delta(S).$$
Then we can adapt~\cite[Lemma 2.1]{CP20} as
\begin{lem}\label{cambiogruppo}
Let $G_1$ and $G_2$ be groups.
Given a subset $S$ of $G_1\setminus \{0_{G_1}\}$ of size $k$, suppose there exists a homomorphism
$\phi: G_1\to G_2$ such that $\ker(\phi)\cap \Upsilon(S)=\emptyset$.
Then $\phi(S)$ is a subset of $G_2\setminus\{0_{G_2}\}$ of size $k$ and the subset $S$ is sequenceable whenever $\phi(S)$ is sequenceable.
\end{lem}

In the following, we consider groups of the form $(G\rtimes_{\varphi} H,\cdot)$ where $(G,+)$ is abelian, $H=\{h_0=0_H,h_1,\dots,h_{e-1}\}$ is finite, $(H,\cdot_H)$\footnote{Also in this section, when it is clear from the context, the index $H$ will be omitted.} is not necessarily abelian and where $\varphi$ is a group homomorphism $\varphi: H\rightarrow Aut(G)$. We will also assume that the homomorphism $\varphi$ satisfies the following property:
\begin{itemize}
\item[$(*)$]\  given $h\in H, \varphi(h)$ is either the identity or $\varphi(h)(x)=-x$ for any $x\in G$.
\end{itemize}
Note that the latter is an automorphism of $G$ because $G$ is abelian.
\begin{rem}
Two classes of semidirect product that satisfy property $(*)$ are the generalized dihedral groups and the standard direct product.
Indeed:
\begin{itemize}
\item[1)] The generalized dihedral group $Dih(G)$ can be defined as $G\rtimes_{\varphi} \Z_2$ where
$$\varphi(0)=id \mbox{ and }\varphi(1)(x)=-x.$$
\item[2)] Any direct product $G\times H$ can be defined as $G\rtimes_{\varphi} H$ where
$$\varphi(h)=id \mbox{ for any }h\in H.$$
\end{itemize}
\end{rem}
Given a multiset $S$ of elements of $G\rtimes_{\varphi} H$, we can define the {\em type} of~$S$ also when $G$ is a generic group. It is the sequence $\bm{\lambda} = (\lambda_0, \ldots, \lambda_{e-1} )$, where $\lambda_i$ is the number of times that $h_i$ appears in $\pi_2(S)$ and $\pi_2$ is the projection over $H$. Note that two multisets $S_1$ of $G_1\rtimes_{\varphi} H$ and $S_2$ of $G_2\rtimes_{\varphi} H$ can have the same type even though $G_1$ is not equal to $G_2$.

Given a semidirect product $G_1\rtimes_{\varphi} H$, we can define the corresponding product $G_2\rtimes_{\varphi'} H$ where $\varphi'(h)$ is the identity on $G_2$ (resp $\varphi'(h)(x)=-x$ for any $x\in G_2$) if and only if $\varphi(h)$ is the identity on $G_1$ (resp $\varphi(h)(x)=-x$ for any $x\in G_1$). For simplicity, we will denote, with abuse of notation, by $G_2\rtimes_{\varphi} H$ this semidirect product.
\begin{prop}\label{prop:Z}
Let $H$ be a finite group, let $\varphi$ satisfy property $(*)$, and suppose that, for infinitely many primes $p$, any subset of a given type $\bm{\lambda} = (\lambda_0, \ldots, \lambda_{e-1} )$ of $\Z_p\rtimes_{\varphi} H\setminus \{0_{\Z_p\rtimes_{\varphi}H}\}$ is sequenceable. Then also any subset of type $\bm{\lambda}$ of $\Z\rtimes_{\varphi} H\setminus \{0_{\Z\rtimes_{\varphi} H}\}$ is sequenceable.
\end{prop}
\begin{proof}
Consider a subset $S$ of $\Z\rtimes_{\varphi} H\setminus \{0_{\Z\rtimes_{\varphi} H}\}$ of type $\bm{\lambda}$. Let $p>\max\limits_{z\in \Upsilon(S)} |\pi_{\Z}(z)|$ be a prime such that any subset of $\Z_p\rtimes_{\varphi} H\setminus \{0_{\Z_p\rtimes_{\varphi} H}\}$ of type $\bm{\lambda}$ is sequenceable.

Here we note that, since $\Z\rtimes_{\varphi} H$ and $\Z_p\rtimes_{\varphi} H$ are defined with corresponding maps $\varphi$, the map $\pi_p\times id: \Z\rtimes_{\varphi} H\rightarrow \Z_p\rtimes_{\varphi} H$ is a group homomorphism.
Indeed in this case we have that
$$(\pi_p\times id)(x_1,k_1)\cdot (\pi_p\times id)(x_2,k_2)=(x_1 \pmod{p},k_1)\cdot (x_2\pmod{p},k_2)=$$
$$=((x_1 + \varphi(k_1)(x_2))\pmod{p} ,k_1\cdot_H k_2)=(\pi_p\times id)[(x_1,k_1)\cdot (x_2,k_2)].$$
Then $\Upsilon(S)$ and the kernel of the map $\pi_p\times id: \Z\rtimes_{\varphi} H\to \Z_p\rtimes_{\varphi} H$ are disjoint sets.
Therefore, since $\pi_p\times id$ is the identity on $H$ and $S$ is of type $\bm{\lambda}$, we have $(\pi_p\times id)(S)$ also is of type $\bm{\lambda}$. It then follows from Lemma~\ref{cambiogruppo} that~$S$ also is sequenceable.
\end{proof}
\begin{rem}
In the previous proposition, we can see why property $(*)$ is needed in this discussion. Indeed we need that $\varphi(a)$ is a group automorphism of $\Z$ and hence it maps $1$ in a generator of $\Z$. It follows that $\varphi(a)(1)\in \{1,-1\}$ or, equivalently, $\varphi$ satisfies property $(*)$.
\end{rem}

We now consider the free abelian group $\Z^n$ of rank $n$. We obtain the following proposition whose proof is omitted since it is the same as Proposition 4.3 of \cite{CDOR}.

\begin{prop}\label{prop:Zn}
Let $H$ be a finite group, let $\varphi$ satisfy property $(*)$, and suppose that any subset of a given type $\bm{\lambda} = (\lambda_0, \ldots, \lambda_{e-1} )$ of $\Z\rtimes_{\varphi} H\setminus\{0_{\Z^n\rtimes_{\varphi} H}\}$ is sequenceable.
Then also any subset of type $\bm{\lambda}$ of $\Z^n\rtimes_{\varphi} H\setminus\{0_{\Z^n\rtimes_{\varphi} H}\}$ (for any $n\geq 2$) is sequenceable.
\end{prop}
From the previous proposition, we deduce this result. Also in this case we omit the proof since it is the same as Theorem 4.4 of \cite{CDOR}

\begin{thm}\label{torsionfree}
Let $H$ be a finite group, let $\varphi$ satisfy property $(*)$, and suppose that, for infinitely many primes $p$, any subset of a given type $\bm{\lambda} = (\lambda_0, \ldots, \lambda_{e-1} )$ of $\Z_p\rtimes_{\varphi} H\setminus\{0_{\Z_p\rtimes_{\varphi} H}\}$ is sequenceable. Then also any subset of $G\rtimes_{\varphi} H\setminus\{0_{G\rtimes_{\varphi} H}\}$ of type $\bm{\lambda}$ is sequenceable provided that $G$ is torsion-free and abelian.
\end{thm}

\begin{cor}
Let $H$ be a finite group, let $\varphi$ satisfy property $(*)$, $k$ be a positive integer, and suppose that, for infinitely many primes $p$, any subset of size $k$ of $\Z_p\rtimes_{\varphi} H\setminus\{0_{\Z_p\rtimes_{\varphi} H}\}$ is sequenceable. Then also any subset of size $k$ of $G\rtimes_{\varphi} H\setminus \{0_{G\rtimes_{\varphi} H}\}$ is sequenceable provided that $G$ is torsion-free and abelian.
\end{cor}
Given an element $g$ of an abelian group $G$, we denote by $o(g)$ the cardinality of the cyclic subgroup $\langle g
\rangle$ generated by $g$. Furthermore, we set
$$\vartheta(G)=\min_{0_G\neq g \in G} o(g).$$

As a consequence of Theorem \ref{torsionfree}, and with essentially the same proof of Theorem 4.6 of \cite{CDOR}, we can prove the following result.
\begin{thm}\label{thm:as}
Let $H$ be a finite group, let $\varphi$ satisfy property $(*)$, and suppose that, for infinitely many primes $p$, any subset of $\Z_p\rtimes_{\varphi} H\setminus \{0_{\Z_p\rtimes_{\varphi} H}\}$ of a given type $\bm{\lambda} = (\lambda_0, \ldots, \lambda_{e-1} )$ is sequenceable.
Then there exists a positive integer $N(\bm{\lambda})$ such that any subset of type $\bm{\lambda}$ of $G\rtimes_{\varphi} H\setminus\{0_{G\rtimes_{\varphi} H}\}$ is sequenceable provided that $G$ is abelian and $\vartheta(G)>N(\bm{\lambda})$.
\end{thm}
\begin{cor}
Let $H$ be a finite group, let $\varphi$ satisfy property $(*)$, $k$ be a positive integer, and suppose that, for infinitely many primes $p$, any subset of size $k$ of $\Z_p\rtimes_{\varphi} H\setminus\{0_{\Z_p\rtimes_{\varphi} H}\}$ is sequenceable. Then, there exists a positive integer $N(k)$ such that any subset of size $k$ of $G\rtimes_{\varphi} H\setminus\{0_{G\rtimes_{\varphi} H}\}$ is sequenceable provided that $G$ is abelian and $\vartheta(G)>N(k)$.
\end{cor}
Therefore, as a consequence of Theorem \ref{th:main}, and recalling that the only exceptions to that theorem, in case of dihedral groups and $k\leq 12$, are for $k=4$ and $p=3$ (see Theorem 3.4 of \cite{Ollis}), we can prove our first asymptotic result.
\begin{thm}\label{th:as1}
Let $G$ be a group such that $\vartheta(G)>N(k)$.
Then any subset $S$ of $Dih(G)$ such that $|S|\leq 12$ is sequenceable.
\end{thm}
\begin{rem}\label{rem:linear2}
Because of Remark \ref{rem:linear1}, with essentially the same proof of Theorem \ref{th:as1}, we can prove that if $S$ is a subset of $Dih(G)$ not of type $(k,0)$, $k\leq 12$ and $\vartheta(G)>N(k)$, then $S$ admits a linear sequencing.

This implies that, under the hypothesis of Theorem \ref{th:as1}, $S$ admits a linear sequencing unless $s_k$ is unavoidably the identity.
\end{rem}

\subsection{An Explicit Upper Bound}

Now, in case $G=\mathbb{Z}_m$, and adapting again the procedure of \cite{CDOR}, we provide an explicit upper bound on the numbers $N(\bm{\lambda})$ and $N(k)$.
First, we need to recall some elementary facts about the solutions of a linear system over fields and commutative rings.
\begin{thm}[Corollary of Cramer's Theorem]\label{Cramer}
Let
\begin{equation}\label{sys1}\begin{cases}
m_{1,1}x_1+\dots+ m_{1,k}x_k=b_1\\
\dots\\
m_{l,1}x_1+\dots+ m_{l,k}x_k=b_l
\end{cases} \end{equation}
be a linear system over a commutative ring $R$ whose associated matrix $M=(m_{i,j})$ is an $l\times k$ matrix over $R$. Let $(x_1,\dots, x_k)$ be a solution of the system and let $M'$ be a $k'\times k'$ square, nonsingular (i.e.~$\det(M')$ is invertible in $R$), submatrix of $M$. We assume, without loss of generality, that $M'$ is obtained by considering the first $k'$ rows and the first $k'$ columns of $M$. Then $x_1,\dots, x_k$ satisfy the following relation

$$\begin{bmatrix}
x_{1} \\
x_{2} \\
\vdots \\
x_{k'}
\end{bmatrix} = \begin{bmatrix}
m_{1,1} & m_{1,2} & m_{1,3} & \dots & m_{1,k'} \\
m_{2,1} & m_{2,2} & m_{2,3} & \dots & m_{2,k'} \\
\vdots & \vdots & \vdots & & \vdots \\
m_{k',1} & m_{k',2} & m_{k',3} & \dots & m_{k',k'}
\end{bmatrix} ^{-1} \begin{bmatrix}
b_1- (m_{1,k'+1}x_{k'+1}+\dots + m_{1,k}x_{k})\\
b_2- (m_{2,k'+1}x_{k'+1}+\dots+ + m_{2,k}x_{k})\\
\vdots \\
b_{k'}- (m_{k',k'+1}x_{k'+1}+\dots+ m_{k',k}x_{k})
\end{bmatrix} .$$

Equivalently, set $$ M'_j:=\begin{bmatrix}
m_{1,1} & m_{1,2} &\dots & m_{1,j-1} & b_1- (m_{1,k'+1}x_{k'+1}+\dots + m_{1,k}x_{k}) & \dots & m_{1,k'} \\
m_{2,1} & m_{2,2} &\dots & m_{2,j-1} & b_2- (m_{2,k'+1}x_{k'+1}+\dots+ + m_{2,k}x_{k}) & \dots & m_{2,k'} \\
\vdots & \vdots & \vdots & \vdots & \vdots & & \vdots \\
m_{k',1} & m_{k',2} &\dots & m_{k',j-1}& b_{k'}- (m_{k',k'+1}x_{k'+1}+\dots+ m_{k',k}x_{k})
& \dots & m_{k',k'}
\end{bmatrix}$$
we have that, for any $j\in [1,k']$
\begin{equation}\label{Cramer2}x_j= \frac{\det(M'_j)}{\det(M')}.\end{equation}
Now assume that $R$ is a field and that the system \eqref{sys1} admits at least one solution. Then, given a maximal square nonsingular submatrix $M'$ of $M$, and fixed $x_{k'+1},\dots, x_k$ in $R$, we have that $x_1, x_2, \dots, x_k$ are a solution of the system \eqref{sys1} if and only if $x_1,\dots, x_{k'}$ satisfy the relation \eqref{Cramer2}.
\end{thm}
\begin{rem}\label{linarcombi}
In case the system \eqref{sys1} of Theorem \ref{Cramer} is homogeneous (i.e. $b_1=b_2=\dots=b_l=0$), due to the Laplace expansion of the determinant, $\det(M'_j)$ can be written as a linear combination of $x_{k'+1}, \dots, x_{k}$. More precisely, denoting by $M'_{i,j}$ the matrix obtained by deleting the $i$-th row and the $j$-th column of $M'$, we have that
\begin{equation}\label{linearcombieq}\det(M'_j)=\sum_{i=1}^{k'} (-1)^{j+i+1} \\det(M'_{i,j})(m_{i,k'+1}x_{k'+1}+\dots + m_{i,k}x_{k}).\end{equation}
\end{rem}
Now we are ready to provide, in case of groups of type $\mathbb{Z}_m\rtimes_{\varphi} H$ (where $\varphi$ satisfies property $(*)$), a quantitative version of Theorem \ref{thm:as}.
In the following proof, developing the notations of the previous sections, given a set $S$ of size $k$ and an ordering $\bm{x}=(x_1,\dots,x_k)$ of $S$, we denote by $s_i(\bm{x})$ the partial sum $x_1\cdot x_2\dots x_{i-1}\cdot x_i$ where $s_0(\bm{x})$ is always zero.
\begin{thm}\label{thm:explicitUB}
Let $H$ be a finite group, let $\varphi$ satisfy property $(*)$, and suppose that, for infinitely many primes $p$, any subset of a given type $\bm{\lambda} = (\lambda_0, \ldots, \lambda_{e-1} )$ of $\Z_p\rtimes_{\varphi} H\setminus \{0_{\Z_p\rtimes_{\varphi} H}\}$ is sequenceable. Then also any subset of type $\bm{\lambda}$ of $\mathbb{Z}_m\rtimes_{\varphi} H\setminus\{0_{\mathbb{Z}_m\rtimes_{\varphi} H}\}$ is sequenceable provided that the prime factors of $m$ are all greater than $k!$ where $k=\lambda_0+\lambda_1+\dots+\lambda_{e-1}$.
\end{thm}
\begin{proof}
Let us assume, by contradiction, that there exists a subset $S=\{x_1,\dots,x_k\}$ of type $\bm{\lambda}$ in $\mathbb{Z}_m\rtimes_{\varphi} H\setminus\{0_{\mathbb{Z}_m\rtimes_{\varphi} H}\}$ that is not sequenceable and all the prime factors of $m$ are larger than $k!$. Here $S$ would be a set of type $\bm{\lambda}$ whose elements do not admit an ordering with different partial sums.

This implies that for any possible permutation $\omega\in \Sym(k)$, considered the ordering $$\bm{x}_{\omega}=(x_{\omega(1)},x_{\omega(2)},\dots,x_{\omega(k)}),$$
there exist $h$ and $t$, with $h,t \in [0,k]$, $\{h,t\}\not=\{0,k\}$ and $h\not=t$, such that $s_{h}(\bm{x}_\omega)=s_{t}(\bm{x}_\omega)$, that is, assuming $h<t$, \begin{equation}\label{system} x_{\omega(h+1)}\cdot x_{\omega(h+2)}\dots x_{\omega(t-1)} \cdot x_{\omega(t)}=0.\end{equation}
Therefore we would have a solution $(\pi_{\mathbb{Z}_m}(x_1),\dots,\pi_{\mathbb{Z}_m}(x_k))^T$ to the system of equations derived from \eqref{system} by considering the projection over $\mathbb{Z}_m$.
In the following, for simplicity, we will denote $\pi_{\mathbb{Z}_m}(x_j)$ by $x_j^1$ and, using this notation, we are considering the system:
\begin{equation}\label{system2} x^1_{\omega(h+1)}+ \varphi_{\pi_{H}(x_{\omega(h+1)})}x^1_{\omega(h+2)}+\dots + \varphi_{\pi_{H}(x_{\omega(h+1)})\cdots \pi_{H}(x_{\omega(t-1)})}x^1_{\omega(t)}=0\end{equation}
where, using the notation of the previous sections, and since $\varphi$ satisfies property $(*)$, we set $\varphi(a)(x)=\varphi_a x$ for some constant $\varphi_a\in \{1,-1\}$.
Denoted by $M$ the $(k!)\times k$ matrix of the coefficients of this system, we note that, since $\varphi$ satisfies property $(*)$, the coefficients of $M$ belong to $\{-1,0,1\}$. It follows that the determinant of any square submatrix of $M$ (that is big at most $k\times k$) is a sum of at most $k!$ terms that belong to $\pm 1$. Therefore, if all the prime factors of $m$ are larger than $k!$, a square submatrix $M'$ of $M$ is nonsingular in $\mathbb{Z}_{m}$ if and only if it is nonsingular in $\mathbb{Q}$.

Following the proof of Thorem 4.12 of \cite{CDOR} and assuming the  existence a subset $S=\{x_1,\dots,x_k\}$ of type $\bm{\lambda}$ in $\mathbb{Z}_m\rtimes_{\varphi} H\setminus\{0_{\mathbb{Z}_m\rtimes_{\varphi} H}\}$ that is not sequenceable, we find a subset of the same type in $\mathbb{Q}\rtimes_{\varphi} H\setminus\{0_{\mathbb{Q}\rtimes_{\varphi} H}\}$ that is also not sequenceable. However, because of {\rm Theorem \ref{torsionfree}} and since $\mathbb{Q}$ is a torsion-free abelian group, any subset $\widetilde{S}$ of type $\bm{\lambda}$ of $ G\rtimes_{\varphi} H\setminus\{0_{G\rtimes_{\varphi} H}\}$ is sequenceable. But this is a contradiction and hence the thesis is verified.
\end{proof}
\begin{cor}
Let $H$ be a finite group, let $\varphi$ satisfy property $(*)$, $k$ be a positive integer, and suppose that, for infinitely many primes $p$, any subset of size $k$ of $\Z_p\rtimes_{\varphi} H\setminus\{0_{\Z_p\rtimes_{\varphi} H}\}$ is sequenceable.
Then also any subset of size $k$ of $\mathbb{Z}_m\rtimes_{\varphi} H\setminus\{0_{\mathbb{Z}_m\rtimes_{\varphi} H}\}$ is sequenceable provided that the prime factors of $m$ are all greater than $k!$.
\end{cor}
If we consider direct products of type $\mathbb{Z}_m\times H$, we can improve the result of Theorem \ref{thm:explicitUB}. Indeed, with essentially the same proof of Theorem 22 of \cite{CDOR}, we can state the following.
\begin{thm}
Let $H$ be a finite group, $k$ be a positive integer, and suppose that, for infinitely many primes $p$, any subset of size $k$ of $\Z_p\times H\setminus\{0_{\Z_p\times H}\}$ is sequenceable.
Then also any subset of size $k$ of $\mathbb{Z}_m\times H\setminus\{0_{\mathbb{Z}_m\times H}\}$ is sequenceable provided that the prime factors of $m$ are all greater than $k!/2$.
\end{thm}

As regards dihedral groups, as a consequence of Theorem \ref{th:main}, we can prove Theorem \ref{th:infty}.
\begin{thm}\label{th:infty}
Let $m$ be such that all its prime factors are larger than $k!$.
Then any subset $S$ of $D_{2m}$ such that $|S|\leq 12$ is sequenceable.
\end{thm}
\begin{rem}\label{rem:linear3}
Because of Remark \ref{rem:linear1}, with essentially the same proof of Theorem \ref{th:infty}, we can prove that if $S$ is a subset of $D_{2m}$ not of type $(k,0)$, $k\leq 12$ and all the prime factors of $m$ are larger than $k!$, then $S$ admits a linear sequencing.

This implies that, under the hypothesis of Theorem \ref{th:infty}, $S$ admits a linear sequencing unless $s_k$ is unavoidably the identity.
\end{rem}
\section{Weak Sequenceability}\label{sec:weak}
In this section we consider again subsets $S$ of a group $(G,\cdot)$ of the form $G=\Z_p\rtimes_{\varphi}H$ which is the semidirect product of $(\mathbb{Z}_p,+)$ and another group $(H,\cdot)$. As  in the previous section we indicate by $x\cdot a$ (or by $y\cdot b)$, where $x\in \Z_p$ and $a\in H$, a generic element of $G$.
We investigate the $t$-weak sequenceability of subsets of semidirect products.
Recall that in addition to requiring that $x_i-x_j\not=0$ for $1\leq i<j\leq k$ and $a_i=a_j$, we seek an ordering having no two of its partial sums $s_i, s_j$ equal for $1\leq i<j\leq k$ and $|i-j|\leq t$. Hence, modifying the expression of Equation \eqref{pol1}, we define, for $t<k$, the following polynomial
\begin{equation}\label{pol2}q_{\bm{a}} = \prod_{\substack{1 \leq i < j \leq k \\ a_i = a_j }} (x_j - x_i)
\prod_{\substack{0 \leq i < j \leq k \\ b_i = b_j \\ j-i\leq t \\ j \neq i+1}} (x_{i+1}+ \varphi_{a_{i+1}}x_{i+2}+\varphi_{a_{i+1}\cdot a_{i+2}}x_{i+3}+\dots+\varphi_{a_{i+1}\dots a_{j-1}}x_j).\end{equation}

In this case we have that a set $S=\{x_1\cdot a_1, x_2\cdot a_2, \ldots, x_k\cdot a_k \}\subseteq G\setminus \{0\}$ of size $k$ is $t$-weak sequenceable if there exists an ordering $\bm{x_a} = \left( x_1\cdot a_1, x_2\cdot a_2, \ldots, x_k\cdot a_k \right)$ of its elements such that $q_{\bm{a}}(x_1,\dots, x_k)\not=0$.

Now, given a set $S=\{x_1\cdot a_1, x_2\cdot a_2, \ldots, x_k\cdot a_k \}$ of $k$ elements, the idea is to fix, {\em a priori}, the first $h$ elements $(x_1\cdot a_1, x_2\cdot a_2, \ldots, x_h\cdot a_h)$, where $h$ is not too big, of the ordering in such a way that no of its partial sums $s_i, s_j$ are equal when $|i-j|\leq t$ and $1\leq i<j\leq h$. This can be expressed by requiring that $q_{\bm{a}'}(x_1,\cdots,x_h)\not=0$ where $\bm{a}'=(a_1,\dots,a_h)$ and we show that this can be done due the following Proposition whose proof is inspired by that of Proposition 2.2 of \cite{CD}.
\begin{prop}\label{fix}
Let $S=\{x_1\cdot a_1, x_2\cdot a_2, \ldots, x_k\cdot a_k \}\subseteq G\setminus \{0\}$ be a set of size $k$ and let $h$ and $t$ be positive integers such that $k-h\geq 2(t-1)$ and $h\geq t-1$. Denoted by $ \bar{a}$ the element that appears the most in $\pi_2(S)$, we assume that at least $2(t-1)$ elements of $S$ have $\bar{a}$ as second coordinate. Then there exists an ordering of $h$-elements of $S$ that we denote, up to relabeling, with $(x_1\cdot a_1, x_2\cdot a_2, \ldots, x_h\cdot a_h)$, such that
\begin{itemize}
\item[(a)] $q_{\bm{a}'}(x_1,\cdots,x_h)\not=0$ where $\bm{a}'=(a_1,\dots,a_h)$;
\item[(b)] $a_{h},a_{h-1},\dots, a_{h-t+2}$ are all equals to $\bar{a}$;
\item[(c)] at least $t-1$ elements of $\{x_{h+1}\cdot a_{h+1}, \ldots, x_k\cdot a_k\}$ have $\bar{a}$ as second coordinate.
\end{itemize}
\end{prop}
\begin{proof}
Here we choose a set $S'$ of $2(t-1)$ elements that have $\bar{a}$ as the second coordinate, and we set $S'':=S\setminus S'$.
Now we prove that we can choose, inductively on $\ell$, $(x_1\cdot a_1, x_2\cdot a_2, \ldots, x_{h-t}\cdot a_{h-t})$ in $S''$ and $(x_{h-t+1}\cdot a_{h-t+1}, \ldots, x_h\cdot a_h)$ in $S'$ such that, for any $\ell\leq h$, we have $q_{(a_1,\dots,a_{\ell})}(x_1,\cdots,x_{\ell})\not=0$. Then, properties $(b)$ and $(c)$ will be guaranteed by our choice for $x_1\cdot a_1,\ldots,x_h\cdot a_h$.

BASE CASE: Let $\ell=1$. Since $q_{(a)}(x)=1$ for any $t$ and for any $x\cdot a$ of both $S'$ and $S''$, the statement is realized for $\ell=1$.

INDUCTIVE CASE: Let us assume the statement for $\ell\in \{1,\dots,m\}$ and let us prove it for $\ell=m+1$ where $m+1\leq h$.
Since the statement is true for $\ell=m$, there exists an $m$-tuple $(x_1\cdot a_1,\dots,x_m\cdot a_m)$ such that
$$q_{(a_1,\dots, a_m)}(x_1,\cdots,x_m)\not=0.$$
Given $x\cdot a\in G$, we set $s_m:=(x_1\cdot a_1)\cdots(x_m\cdot a_m)=y_m\cdot b_m$ and we define $y$ and $s$ such that $s=s_m\cdot (x\cdot a)=y\cdot b$ where $b=b_m\cdot a$. Then we have
$$\frac{q_{(a_1,\dots, a_m,a)}(x_1,\dots,x_m,x)}{q_{(a_1,\dots, a_m)}(x_1,\dots,x_m)}=$$
$$=\prod_{1\leq i<m+1,\ a=a_i} (x-x_i)\prod_{\substack{0\leq i<m \\ m+1-i\leq t,\ b_i=b}} (x_{i+1}+ \varphi_{a_{i+1}}x_{i+2}+\varphi_{a_{i+1}\cdot a_{i+2}}x_{i+3}+\dots+\varphi_{a_{i+1}\dots a_{m}}x).$$
Here, any element $x\cdot a$ of $S\setminus \{x_1\cdot a_1,\dots,x_m\cdot a_m\}$ satisfies
$\prod_{1\leq i<m+1} (x-x_i)\not=0$ if $a=a_i$. Hence, to have $\frac{q_{(a_1,\dots, a_m)}(x_1,\dots,x_m,x)}{q_{(a_1,\dots, a_m,a)}(x_1,\dots,x_m)}\not=0$, it suffice to find $x$ such that
$$\prod_{\max(0,m+1-t)\leq i<m,\ b_i=b}(x_{i+1}+ \varphi_{a_{i+1}}x_{i+2}+\varphi_{a_{i+1}\cdot a_{i+2}}x_{i+3}+\dots+\varphi_{a_{i+1}\dots a_{m}}x) \neq 0.$$

Note that there is at most one $x\cdot a\in S\setminus\{x_1\cdot a_1,\dots,x_m\cdot a_m\}$ such that
$$(x_{i+1}+ \varphi_{a_{i+1}}x_{i+2}+\varphi_{a_{i+1}\cdot a_{i+2}}x_{i+3}+\dots+\varphi_{a_{i+1}\dots a_{m}}x)=0$$
and $b_i=b$. Indeed this would imply that $s_i=s$ in $G$ that is $(x_{i+1}\cdot a_{i+1})\cdot \ldots \cdot (x_m\cdot a_m)\cdot (x\cdot a)=0$ which is satisfied by exactly one element $x\cdot a$ of $G$. Since those relations are at most $t-1$, we have at most $t-1$ values $x\cdot a$ in $S\setminus\{x_1\cdot a_1,\dots,x_m\cdot a_m\}$ such that $$\prod_{\max(0,m+1-t)\leq i<m,\ b_i=b} (x_{i+1}+ \varphi_{a_{i+1}}x_{i+2}+\varphi_{a_{i+1}\cdot a_{i+2}}x_{i+3}+\dots+\varphi_{a_{i+1}\dots a_{m}}x)=0.$$

Here we have that either $h-t+1>m$ or $h-t+1\leq m$. In the first case, we have that
$$|S''\setminus \{x_1\cdot a_1,\dots, x_m\cdot a_m\}|=k-2(t-1)-m\geq t+1$$
where the inequality holds since $k-2(t-1)-m=(k-m)-2(t-1)$ and $k-m\geq k-h\geq 2(t-1)$.
It follows that there exists $x\cdot a\in S''\setminus \{x_1\cdot a_1,\dots, x_m\cdot a_m\}$ such that
$$\prod_{\max(0,m+1-t)\leq i<m,\ b_i=b} (x_{i+1}+ \varphi_{a_{i+1}}x_{i+2}+\varphi_{a_{i+1}\cdot a_{i+2}}x_{i+3}+\dots+\varphi_{a_{i+1}\dots a_{m}}x)\not=0.$$
In the second case, we have that
$$|S'\setminus \{x_1\cdot a_1,\dots, x_m\cdot a_m\}|=2(t-1)-(m-(h-t+1))\geq t$$
where the last inequality holds since $m\leq h-1$ and hence $2(t-1)-(m-(h-t+1))\geq 2(t-1)-(h-1-(h-t+1))=t$.
It follows that also in this case, there exists $x\cdot a\in S'\setminus \{x_1\cdot a_1,\dots, x_m\cdot a_m\}$ such that
$$\prod_{\max(0,m+1-t)\leq i<m,\ b_i=b}(x_{i+1}+ \varphi_{a_{i+1}}x_{i+2}+\varphi_{a_{i+1}\cdot a_{i+2}}x_{i+3}+\dots+\varphi_{a_{i+1}\dots a_{m}}x)\not=0.$$

In both cases, we can find $x_{m+1}\cdot a_{m+1}$ that satisfies property $(b)$ and $(c)$ and which it is such that
$$\frac{q_{(a_1,\dots, a_m,a_{m+1})}(x_1,\dots,x_m,x_{m+1})}{q_{(a_1,\dots, a_m)}(x_1,\dots,x_m)}=\prod_{1\leq i<m+1,\ a_{m+1}=a_i} (x_{m+1}-x_i)\cdot$$
$$\cdot\prod_{\max(0,m+1-t)\leq i<m,\ b_i=b_{m+1}}(x_{i+1}+ \varphi_{a_{i+1}}x_{i+2}+\varphi_{a_{i+1}\cdot a_{i+2}}x_{i+3}+\dots+\varphi_{a_{i+1}\dots a_{m}}x_{m+1})\not=0.$$
Since $\mathbb{Z}_p$ is a field and due to the inductive hypothesis $q_{(a_1,\dots, a_m,a)}(x_1,\dots,x_m) \neq 0$, we also have that
\begin{multline}
\frac{q_{(a_1,\dots, a_m,a_{m+1})}(x_1,\dots,x_m,x_{m+1})}{q_{(a_1,\dots, a_m)}(x_1,\dots,x_m)}\cdot  \\ q_{(a_1,\dots, a_m)}(x_1,\dots,x_m)=q_{(a_1,\dots, a_m,a_{m+1})}(x_1,\dots,x_{m+1})\not=0\end{multline}
which completes the proof.
\end{proof}
\begin{rem}\label{bound}
Note that the hypothesis that at least $2(t-1)$ elements of $S$ has the same second coordinate is always satisfied when
$$k>(2t-3)\frac{|G|}{p}=(2t-3)|H|.$$
This means that the approach of Proposition \ref{fix} can be always tried assuming that $k$ is big enough. For smaller values of $k$ one could fix the first elements differently, as done for example in Proposition 2.2 of \cite{CD}, but, since then one should consider a huge number of cases (see also the following Remark \ref{MagicRemark}), we consider explicitly only this asymptotic case.
\end{rem}
In the following we assume that we have fixed, according to Proposition \ref{fix}, $$\{x_1\cdot a_1, x_2\cdot a_2, \ldots, x_h\cdot a_h\}\subseteq S$$ that satisfies properties $(a),(b)$ and $(c)$ of such proposition. We note that every $x\cdot a \in S\setminus\{x_1 \cdot a_1, x_2\cdot a_2, \ldots, x_h\cdot a_h\}$ is such that $x-x_i\not=0$ for any $i\in \{1,\dots,h\}$ such that $a=a_i$. Therefore, if $\bm{a}$ is a quotient sequencing that extends $\bm{a}''$, it is left to find a nonzero point for the polynomial
$$\frac{q_{\bm{a}}(x_1,\dots,x_h,x_{h+1},\dots,x_{k})}{q_{\bm{a}'}(x_1,\cdots,x_h)\prod_{1\leq i\leq h<j\leq k,\ a_j=a_i }(x_j-x_i)}.$$
Since the free variables are now $x_{h+1},\dots,x_k$, we set $\ell:=k-h$ and $z_i := x_{i+h}$; here the constrain $k-h\geq 2(t-1)$ of Proposition \ref{fix} becomes $\ell\geq 2(t-1)$.
Then we denote by $h_{\bm{a},\ell}$ the polynomial
\begin{equation}\label{definizioneH}h_{\bm{a},\ell}(z_1,\dots,z_{\ell}):=\frac{q_{\bm{a}}(x_1,\dots,x_h,z_{1},\dots,z_{\ell})}{q_{\bm{a}'}(x_1,\cdots,x_h)\prod_{1\leq i\leq 1<j\leq \ell,\ a_j=a_i }(z_j-x_i)}.\end{equation} Since we are also assuming that $k-\ell\geq t-1$, that is, $k-(t-1)\geq \ell\geq 2(t-1)$, denoting by $\bm{a}''$ the quotient sequencing given by the last $\ell$ elements of $\bm{a}$, we obtain the following expression
\begin{multline}\label{espressioneH}
h_{\bm{a},\ell}(z_1,\cdots,z_{\ell})=q_{\bm{a}''}(z_1,\cdots,z_{\ell}) \cdot \\
\prod_{\substack{0\leq i\leq t-1;\\ 1\leq j\leq t-i-1;\\ b_{k-\ell-i-1}=b_{(k-\ell)+j}}} (x_{k-\ell-i}+\varphi_{\bar{a}}x_{k-\ell-i+1}+\cdots+\varphi_{\bar{a}^i}x_{k-\ell} +  \varphi_{\bar{a}^{i+1}}z_1+\cdots+\varphi_{\bar{a}^{i+1}\cdot a_{k-\ell+1}\cdot \dots \cdot a_{k-\ell+j-1}}z_j).
\end{multline}
Now we aim to apply the Non-Vanishing Corollary (of the Combinatorial Nullstellensatz) to the polynomial $h_{\bm{a},\ell}$. For this purpose, it is enough to consider the terms of $h_{\bm{a},\ell}$ of maximal degree in the variables $z_1,\dots,z_{\ell}$ that are the ones where no $x_i$ appears. We denote by $r_{\bm{a},\ell}$ the polynomial given by those terms, that is
\begin{multline}\label{espressioneQ}r_{\bm{a},\ell}:=q_{\bm{a}''}(z_1,\cdots,z_{\ell}) \cdot \\ \prod_{\substack{0\leq i\leq t-1;\\ 1\leq j\leq t-i-1;\\ b_{k-\ell-i-1}=b_{(k-\ell)+j}}} (\varphi_{\bar{a}^{i+1}}z_1+\varphi_{\bar{a}^{i+1}\cdot a_{k-\ell+1}}z_2+\cdots+\varphi_{\bar{a}^{i+1}\cdot a_{k-\ell+1}\cdot \dots \cdot a_{k-\ell+j-1}}z_j).
\end{multline}
Since $b_{k-\ell-i-1}=b_{(k-\ell)+j}$ if and only if
$$a_{k-\ell-i} \cdot a_{k-\ell-1}\cdot \ldots \cdot a_{k-\ell}\cdot a_{k-\ell+1}\cdot\ldots\cdot a_{k-\ell+j}= id$$
and $a_{k-\ell-i}= a_{k-\ell-1}= \ldots= a_{k-\ell}=\bar{a}$,
we can state the following, simple but very powerful, remark.
\begin{rem}\label{MagicRemark}
Since we are assuming that $a_{h},a_{h-1},\dots, a_{h-t+1}$ are all equals to $\bar{a}$, the expression of $r_{\bm{a},\ell}$ does not depend on all the quotient sequence $\bm{a}$ but only by $\bm{a}''$ and $\bar{a}$. Here property $b$ of Proposition \ref{fix} is crucial since, otherwise, given $\bm{a}''$, we would have $|H|^{t-1}$ possibilities for the polynomial $r_{\bm{a},\ell}(z_1,\cdots,z_{\ell})$ according to the sequence $a_{h},a_{h-1},$ $\ldots, a_{h-t+1}$.

In the following, we just denote this polynomial by $r_{\bm{a}'',\bar{a},\ell}$.
\end{rem}
Indeed Remark \ref{MagicRemark} means that, after these manipulations, we are left to consider a polynomial that does not depend on $k=|S|$, and hence we have chances to get a result that is very general on $k$.

To apply the Non-Vanishing Corollary, we also need to find, for any quotient sequencing $\bm{a}''$ of length $\ell$, a nonzero coefficient in $r_{\bm{a}'',\bar{a},\ell}$ that divides the bounding monomial. Moreover, due to property $c$ of Proposition \ref{fix}, we may assume that $\bar{a}$ appears at least $t-1$ times in $\bm{a}''$.
\subsection{Computational Results}

As done for the sequenceability problem in the more general setting, we consider here groups of type $D_{2p}$ and $G_{3p}$. Also, since the case $G=\Z_p\times \Z_e$ has not been considered in \cite{CD}, we consider these groups here.
\begin{ex}{Example}
Let $t = 4$,  $\ell = 2(t-1) = 6$ and $\bar{a} = 0$.  Let $p$ be a prime and suppose~$S \subseteq D_{2p} \setminus \{0\}$ with $|S| = k \geq 11$ and suppose that $S$ has at least $6$ elements with coset $\bar{a}$.  Let $S' = S \setminus \{x_1 \cdot a_1 , \ldots,  x_{k-\ell} \cdot a_{k - \ell}\}$ be of type $(3,3)$.  The sequence $\bm{a''} = (1, 0, 1, 0, 1, 0)$ has partial sums $(0, 1, 1, 0, 0, 1, 1)$.

By Proposition \ref{fix},  we desire a sequencing of~$S'$ of the form 
$$\left( z_1 \cdot 1,  z_2 \cdot 0,  z_3 \cdot 1,  z_4 \cdot 0,  z_5 \cdot 1,  z_6 \cdot 0  \right).$$
The polynomial we need to study is the one defined in \eqref{espressioneQ}.  Hence we get
\begin{align*}
&r_{\bm{a}'',\bar{a},\ell} = (z_1 - z_2 - z_3) (z_3 - z_1) (z_5 - z_1) (z_4 - z_2) (z_6 - z_2) (z_5 - z_3) (z_6 - z_4)  \\ & (z_1 - z_2 - z_3) (z_1 - z_2 -z_3 + z_4) (z_2 + z_3 - z_4 - z_5) (z_3 - z_4 - z_5)(z_3-z_4-z_5+z_6)
\end{align*}
To apply the Non-Vanishing Corollary we need a monomial of this polynomial which divides the bounding monomial $z_1^2 z_2^2 z_3^2 z_4^2 z_5^2 z_6^2$ with a nonzero coefficient.  One such is $z_1^2 z_2^2 z_3^2 z_4^2 z_5^2 z_6^2$,  which has coefficient~$-12$.   Hence whenever~$S$ has this form it is $t$-weakly sequenceable. 
\end{ex}
We are now ready to prove the main result of this section.
\begin{thm}\label{th:main2}
Let $n = pe$ with~$p>3$ prime and let $G$ be a group of size $n$. Then subsets~$S$ of size~$k$ of~$G$ are $t$-weakly sequenceable whenever $e\in \{1,2,3\}$, $k$ is large enough and $t\leq 6$.
\end{thm}
\begin{proof}
The case $e=1$ follows from the results of \cite{CD}.

As already noted in the proof of Theorem \ref{th:main}, a group of size $2p$ is either $\Z_p\times \Z_2$ or the dihedral group $D_{2p}$ and a group of size $3p$ is either $\Z_p\times \Z_3$ or $G_{3p}$ for some value of $r$.
Here, due to Remark \ref{bound}, if $k\geq (2t-3)e+1$, $S$ admits a coset whose size is at least $2(t-1)$ and hence, due to Proposition \ref{fix}, we can apply Non-Vanishing Corollary to the polynomial of Equation \eqref{espressioneQ} for suitable quotient sequences and nonzero monomial (see Tables \ref{tab:6_Prod},  \ref{tab:6_Prod1},  \ref{tab:6_D} and \ref{tab:6_G} below). These results are obtained using SageMath~\cite{SageMath}. 

Here we note that the coefficients presented in Table \ref{tab:6_G} can be zero only for finitely many primes $p$ (see also the discussion in the proof of Theorem \ref{th:main}). On the other hand, since $3p>k$, if $k$ is large enough all these coefficients must be nonzero.
\end{proof}

The following Tables~\ref{tab:6_Prod},  \ref{tab:6_Prod1},  \ref{tab:6_D} and \ref{tab:6_G} contain the required monomials and their coefficients for the proof of Theorem~\ref{th:main2} in the cases~$G\in \{\Z_p\times \Z_2, \Z_p\times \Z_3,D_{2p},G_{3p}\}$.
\footnotesize
\begin{longtable}{llllll}
\caption{Monomials and their coefficients sufficient for the proof of Theorem~\ref{th:main2} in the case $t=6$ and~$G=\Z_p\times \Z_2$.}\label{tab:6_Prod}\\
\hline
$\bar{a}$ & $\bm{\lambda}$ & $\mathbf{a''}$ & deg & monomial/s & coefficient/s \\
\hline
\endfirsthead
\hline
$\bar{a}$ & $\bm{\lambda}$ & $\mathbf{a''}$ & deg & monomial/s & coefficient/s \\
\hline
\endhead
$0$ & $(5, 5)$ & \begin{tabular}{@{}l@{}} $(0, 1, 0, 0, 1, 1, 1, 0, 0, 1)$ \end{tabular} & \begin{tabular}{@{}l@{}} 40 \end{tabular} & \begin{tabular}{@{}l@{}} $x_{1}^{4}x_{2}^{4}x_{3}^{4}x_{4}^{4}x_{5}^{4}x_{6}^{4}x_{7}^{4}x_{8}^{4}x_{9}^{4}x_{10}^{4}$ \end{tabular} & \begin{tabular}{@{}l@{}} $- 5 \cdot 467$ \end{tabular}\\ \hline \\ 
$0$ & $(6, 4)$ & \begin{tabular}{@{}l@{}} $(1, 0, 0, 0, 1, 0, 1, 0, 1, 0)$ \end{tabular} & \begin{tabular}{@{}l@{}} 38 \end{tabular} & \begin{tabular}{@{}l@{}} $x_{2}^{4}x_{3}^{5}x_{4}^{5}x_{5}^{3}x_{6}^{5}x_{7}^{3}x_{8}^{5}x_{9}^{3}x_{10}^{5}$ \\ $x_{1}x_{2}^{3}x_{3}^{5}x_{4}^{5}x_{5}^{3}x_{6}^{5}x_{7}^{3}x_{8}^{5}x_{9}^{3}x_{10}^{5}$ \end{tabular} & \begin{tabular}{@{}l@{}} $- 1063$ \\ $2^4 \cdot 503$ \end{tabular}\\ \hline \\ 
$0$ & $(7, 3)$ & \begin{tabular}{@{}l@{}} $(1, 0, 0, 0, 1, 0, 0, 0, 1, 0)$ \end{tabular} & \begin{tabular}{@{}l@{}} 36 \end{tabular} & \begin{tabular}{@{}l@{}} $x_{3}^{2}x_{4}^{6}x_{5}^{2}x_{6}^{6}x_{7}^{6}x_{8}^{6}x_{9}^{2}x_{10}^{6}$ \end{tabular} & \begin{tabular}{@{}l@{}} $2$ \end{tabular}\\ \hline \\ 
$0$ & $(8, 2)$ & \begin{tabular}{@{}l@{}} $(1, 0, 0, 0, 1, 0, 0, 0, 0, 0)$ \end{tabular} & \begin{tabular}{@{}l@{}} 45 \end{tabular} & \begin{tabular}{@{}l@{}} $x_{3}^{2}x_{4}^{7}x_{5}x_{6}^{7}x_{7}^{7}x_{8}^{7}x_{9}^{7}x_{10}^{7}$ \end{tabular} & \begin{tabular}{@{}l@{}} $- 2^2$ \end{tabular}\\ \hline \\ 
$0$ & $(9, 1)$ & \begin{tabular}{@{}l@{}} $(0, 0, 0, 0, 1, 0, 0, 0, 0, 0)$ \end{tabular} & \begin{tabular}{@{}l@{}} 66 \end{tabular} & \begin{tabular}{@{}l@{}} $x_{1}^{5}x_{2}^{7}x_{3}^{7}x_{4}^{8}x_{6}^{7}x_{7}^{8}x_{8}^{8}x_{9}^{8}x_{10}^{8}$ \end{tabular} & \begin{tabular}{@{}l@{}} $- 2^5 \cdot 3$ \end{tabular}\\ \hline \\ 
$1$ & $(4, 6)$ & \begin{tabular}{@{}l@{}} $(0, 1, 1, 1, 0, 0, 1, 0, 1, 1)$ \end{tabular} & \begin{tabular}{@{}l@{}} 42 \end{tabular} & \begin{tabular}{@{}l@{}} $x_{1}^{3}x_{2}^{5}x_{3}^{5}x_{4}^{5}x_{5}^{3}x_{6}^{3}x_{7}^{5}x_{8}^{3}x_{9}^{5}x_{10}^{5}$ \end{tabular} & \begin{tabular}{@{}l@{}} $- 1187$ \end{tabular}\\ \hline \\ 
$1$ & $(3, 7)$ & \begin{tabular}{@{}l@{}} $(1, 0, 1, 1, 1, 1, 1, 0, 1, 0)$ \end{tabular} & \begin{tabular}{@{}l@{}} 47 \end{tabular} & \begin{tabular}{@{}l@{}} $x_{1}^{5}x_{2}^{2}x_{3}^{6}x_{4}^{6}x_{5}^{6}x_{6}^{6}x_{7}^{6}x_{8}^{2}x_{9}^{6}x_{10}^{2}$ \\ $x_{1}^{6}x_{2}x_{3}^{6}x_{4}^{6}x_{5}^{6}x_{6}^{6}x_{7}^{6}x_{8}^{2}x_{9}^{6}x_{10}^{2}$ \end{tabular} & \begin{tabular}{@{}l@{}} $2 \cdot 3^2 \cdot 1511$ \\ $23 \cdot 499$ \end{tabular}\\ \hline \\ 
$1$ & $(2, 8)$ & \begin{tabular}{@{}l@{}} $(1, 0, 1, 1, 1, 1, 1, 1, 1, 0)$ \end{tabular} & \begin{tabular}{@{}l@{}} 54 \end{tabular} & \begin{tabular}{@{}l@{}} $x_{1}^{4}x_{2}x_{3}^{6}x_{4}^{7}x_{5}^{7}x_{6}^{7}x_{7}^{7}x_{8}^{7}x_{9}^{7}x_{10}$ \\ $x_{1}^{5}x_{2}x_{3}^{5}x_{4}^{7}x_{5}^{7}x_{6}^{7}x_{7}^{7}x_{8}^{7}x_{9}^{7}x_{10}$ \end{tabular} & \begin{tabular}{@{}l@{}} $- 2^5 \cdot 5^2$ \\ $- 4337$ \end{tabular}\\ \hline \\ 
$1$ & $(1, 9)$ & \begin{tabular}{@{}l@{}} $(1, 0, 1, 1, 1, 1, 1, 1, 1, 1)$ \end{tabular} & \begin{tabular}{@{}l@{}} 62 \end{tabular} & \begin{tabular}{@{}l@{}} $x_{1}^{5}x_{3}^{3}x_{4}^{6}x_{5}^{8}x_{6}^{8}x_{7}^{8}x_{8}^{8}x_{9}^{8}x_{10}^{8}$ \end{tabular} & \begin{tabular}{@{}l@{}} $- 2$ \end{tabular}\\ \hline \\ 
$1$ & $(0, 10)$ & \begin{tabular}{@{}l@{}} $(1, 1, 1, 1, 1, 1, 1, 1, 1, 1)$ \end{tabular} & \begin{tabular}{@{}l@{}} 75 \end{tabular} & \begin{tabular}{@{}l@{}} $x_{1}^{3}x_{2}^{4}x_{3}^{5}x_{4}^{9}x_{5}^{9}x_{6}^{9}x_{7}^{9}x_{8}^{9}x_{9}^{9}x_{10}^{9}$ \\ $x_{1}^{3}x_{2}^{5}x_{3}^{5}x_{4}^{8}x_{5}^{9}x_{6}^{9}x_{7}^{9}x_{8}^{9}x_{9}^{9}x_{10}^{9}$ \end{tabular} & \begin{tabular}{@{}l@{}} $- 2^4 \cdot 5^2$ \\ $- 7 \cdot 853$ \end{tabular}\\ \hline
\end{longtable}

\footnotesize
\begin{longtable}{llllll}
\caption{Monomials and their coefficients sufficient for the proof of Theorem~\ref{th:main2} in the case $t=6$ and~$G=\Z_p\times \Z_3$.}\label{tab:6_Prod1}\\
\hline
$\bar{a}$ & $\bm{\lambda}$ & $\mathbf{a''}$ & deg & monomial/s & coefficient/s \\
\hline
\endfirsthead
\hline
$\bar{a}$ & $\bm{\lambda}$ & $\mathbf{a''}$ & deg & monomial/s & coefficient/s \\
\hline
\endhead
$0$ & $(5, 3, 2)$ & \begin{tabular}{@{}l@{}} $(1, 0, 0, 1, 0, 0, 0, 2, 1, 2)$ \end{tabular} & \begin{tabular}{@{}l@{}} 25 \end{tabular} & \begin{tabular}{@{}l@{}} $x_{1}^{2}x_{2}x_{3}^{4}x_{4}^{2}x_{5}^{4}x_{6}^{4}x_{7}^{4}x_{8}x_{9}^{2}x_{10}$ \end{tabular} & \begin{tabular}{@{}l@{}} $2 \cdot 3^2$ \end{tabular}\\ \hline \\ 
$0$ & $(5, 4, 1)$ & \begin{tabular}{@{}l@{}} $(1, 0, 0, 1, 0, 0, 0, 1, 1, 2)$ \end{tabular} & \begin{tabular}{@{}l@{}} 22 \end{tabular} & \begin{tabular}{@{}l@{}} $x_{3}^{4}x_{4}x_{5}^{3}x_{6}^{4}x_{7}^{4}x_{8}^{3}x_{9}^{3}$ \end{tabular} & \begin{tabular}{@{}l@{}} $1$ \end{tabular}\\ \hline \\ 
$0$ & $(5, 5, 0)$ & \begin{tabular}{@{}l@{}} $(0, 1, 0, 1, 0, 0, 0, 1, 1, 1)$ \end{tabular} & \begin{tabular}{@{}l@{}} 33 \end{tabular} & \begin{tabular}{@{}l@{}} $x_{1}^{2}x_{2}x_{3}^{4}x_{4}^{4}x_{5}^{4}x_{6}^{4}x_{7}^{4}x_{8}^{2}x_{9}^{4}x_{10}^{4}$ \end{tabular} & \begin{tabular}{@{}l@{}} $3 \cdot 5^2$ \end{tabular}\\ \hline \\ 
$0$ & $(6, 2, 2)$ & \begin{tabular}{@{}l@{}} $(1, 0, 0, 1, 0, 0, 0, 0, 2, 2)$ \end{tabular} & \begin{tabular}{@{}l@{}} 25 \end{tabular} & \begin{tabular}{@{}l@{}} $x_{3}^{3}x_{4}x_{5}^{4}x_{6}^{5}x_{7}^{5}x_{8}^{5}x_{9}x_{10}$ \end{tabular} & \begin{tabular}{@{}l@{}} $-1$ \end{tabular}\\ \hline \\ 
$0$ & $(6, 3, 1)$ & \begin{tabular}{@{}l@{}} $(1, 0, 0, 1, 0, 0, 0, 0, 1, 2)$ \end{tabular} & \begin{tabular}{@{}l@{}} 30 \end{tabular} & \begin{tabular}{@{}l@{}} $x_{2}x_{3}^{5}x_{4}^{2}x_{5}^{5}x_{6}^{5}x_{7}^{5}x_{8}^{5}x_{9}^{2}$ \end{tabular} & \begin{tabular}{@{}l@{}} $- 2^4$ \end{tabular}\\ \hline \\ 
$0$ & $(6, 4, 0)$ & \begin{tabular}{@{}l@{}} $(1, 0, 0, 1, 0, 0, 0, 0, 1, 1)$ \end{tabular} & \begin{tabular}{@{}l@{}} 28 \end{tabular} & \begin{tabular}{@{}l@{}} $x_{3}^{3}x_{4}x_{5}^{4}x_{6}^{5}x_{7}^{5}x_{8}^{5}x_{9}^{2}x_{10}^{3}$ \end{tabular} & \begin{tabular}{@{}l@{}} $-1$ \end{tabular}\\ \hline \\ 
$0$ & $(7, 2, 1)$ & \begin{tabular}{@{}l@{}} $(1, 0, 0, 1, 0, 0, 0, 0, 0, 2)$ \end{tabular} & \begin{tabular}{@{}l@{}} 33 \end{tabular} & \begin{tabular}{@{}l@{}} $x_{3}^{3}x_{4}x_{5}^{5}x_{6}^{6}x_{7}^{6}x_{8}^{6}x_{9}^{6}$ \end{tabular} & \begin{tabular}{@{}l@{}} $- 2^2$ \end{tabular}\\ \hline \\ 
$0$ & $(7, 3, 0)$ & \begin{tabular}{@{}l@{}} $(1, 0, 0, 1, 0, 0, 0, 0, 0, 1)$ \end{tabular} & \begin{tabular}{@{}l@{}} 35 \end{tabular} & \begin{tabular}{@{}l@{}} $x_{3}^{3}x_{4}x_{5}^{5}x_{6}^{6}x_{7}^{6}x_{8}^{6}x_{9}^{6}x_{10}^{2}$ \end{tabular} & \begin{tabular}{@{}l@{}} $- 2^2$ \end{tabular}\\ \hline \\ 
$0$ & $(8, 1, 1)$ & \begin{tabular}{@{}l@{}} $(1, 0, 0, 0, 2, 0, 0, 0, 0, 0)$ \end{tabular} & \begin{tabular}{@{}l@{}} 44 \end{tabular} & \begin{tabular}{@{}l@{}} $x_{3}^{2}x_{4}^{7}x_{6}^{7}x_{7}^{7}x_{8}^{7}x_{9}^{7}x_{10}^{7}$ \end{tabular} & \begin{tabular}{@{}l@{}} $- 2^2$ \end{tabular}\\ \hline \\ 
$0$ & $(8, 2, 0)$ & \begin{tabular}{@{}l@{}} $(1, 0, 0, 0, 1, 0, 0, 0, 0, 0)$ \end{tabular} & \begin{tabular}{@{}l@{}} 42 \end{tabular} & \begin{tabular}{@{}l@{}} $x_{3}^{2}x_{4}^{5}x_{5}x_{6}^{6}x_{7}^{7}x_{8}^{7}x_{9}^{7}x_{10}^{7}$ \end{tabular} & \begin{tabular}{@{}l@{}} $- 2^2$ \end{tabular}\\ \hline \\ 
$0$ & $(9, 1, 0)$ & \begin{tabular}{@{}l@{}} $(0, 0, 0, 1, 0, 0, 0, 0, 0, 0)$ \end{tabular} & \begin{tabular}{@{}l@{}} 66 \end{tabular} & \begin{tabular}{@{}l@{}} $x_{1}^{5}x_{2}^{6}x_{3}^{8}x_{5}^{7}x_{6}^{8}x_{7}^{8}x_{8}^{8}x_{9}^{8}x_{10}^{8}$ \end{tabular} & \begin{tabular}{@{}l@{}} $2^2 \cdot 7$ \end{tabular}\\ \hline \\ 
$1$ & $(3, 5, 2)$ & \begin{tabular}{@{}l@{}} $(0, 1, 1, 0, 1, 1, 1, 2, 0, 2)$ \end{tabular} & \begin{tabular}{@{}l@{}} 27 \end{tabular} & \begin{tabular}{@{}l@{}} $x_{1}x_{2}^{4}x_{3}^{4}x_{4}^{2}x_{5}^{4}x_{6}^{4}x_{7}^{4}x_{8}x_{9}^{2}x_{10}$ \\ $x_{1}^{2}x_{2}^{3}x_{3}^{4}x_{4}^{2}x_{5}^{4}x_{6}^{4}x_{7}^{4}x_{8}x_{9}^{2}x_{10}$ \end{tabular} & \begin{tabular}{@{}l@{}} $- 73$ \\ $- 233$ \end{tabular}\\ \hline \\ 
$1$ & $(2, 5, 3)$ & \begin{tabular}{@{}l@{}} $(0, 1, 0, 1, 2, 1, 1, 1, 2, 2)$ \end{tabular} & \begin{tabular}{@{}l@{}} 28 \end{tabular} & \begin{tabular}{@{}l@{}} $x_{1}x_{2}^{4}x_{3}x_{4}^{4}x_{5}^{2}x_{6}^{4}x_{7}^{4}x_{8}^{4}x_{9}^{2}x_{10}^{2}$ \end{tabular} & \begin{tabular}{@{}l@{}} $7$ \end{tabular}\\ \hline \\ 
$1$ & $(4, 5, 1)$ & \begin{tabular}{@{}l@{}} $(0, 1, 1, 0, 1, 1, 1, 0, 0, 2)$ \end{tabular} & \begin{tabular}{@{}l@{}} 31 \end{tabular} & \begin{tabular}{@{}l@{}} $x_{1}^{2}x_{2}^{4}x_{3}^{4}x_{4}^{3}x_{5}^{4}x_{6}^{4}x_{7}^{4}x_{8}^{3}x_{9}^{3}$ \\ $x_{1}^{3}x_{2}^{3}x_{3}^{4}x_{4}^{3}x_{5}^{4}x_{6}^{4}x_{7}^{4}x_{8}^{3}x_{9}^{3}$ \end{tabular} & \begin{tabular}{@{}l@{}} $- 5 \cdot 43$ \\ $- 3 \cdot 13 \cdot 17$ \end{tabular}\\ \hline \\ 
$1$ & $(1, 5, 4)$ & \begin{tabular}{@{}l@{}} $(1, 2, 2, 0, 2, 2, 1, 1, 1, 1)$ \end{tabular} & \begin{tabular}{@{}l@{}} 32 \end{tabular} & \begin{tabular}{@{}l@{}} $x_{1}^{4}x_{2}^{3}x_{3}^{3}x_{5}^{3}x_{6}^{3}x_{7}^{4}x_{8}^{4}x_{9}^{4}x_{10}^{4}$ \end{tabular} & \begin{tabular}{@{}l@{}} $- 7$ \end{tabular}\\ \hline \\ 
$1$ & $(0, 5, 5)$ & \begin{tabular}{@{}l@{}} $(1, 2, 1, 2, 1, 1, 1, 2, 2, 2)$ \end{tabular} & \begin{tabular}{@{}l@{}} 38 \end{tabular} & \begin{tabular}{@{}l@{}} $x_{1}^{3}x_{2}^{4}x_{3}^{4}x_{4}^{4}x_{5}^{4}x_{6}^{3}x_{7}^{4}x_{8}^{4}x_{9}^{4}x_{10}^{4}$ \end{tabular} & \begin{tabular}{@{}l@{}} $- 2 \cdot 3 \cdot 5$ \end{tabular}\\ \hline \\ 
$1$ & $(2, 6, 2)$ & \begin{tabular}{@{}l@{}} $(0, 1, 0, 1, 1, 2, 2, 1, 1, 1)$ \end{tabular} & \begin{tabular}{@{}l@{}} 30 \end{tabular} & \begin{tabular}{@{}l@{}} $x_{1}x_{2}^{2}x_{3}x_{4}^{4}x_{5}^{5}x_{6}x_{7}x_{8}^{5}x_{9}^{5}x_{10}^{5}$ \end{tabular} & \begin{tabular}{@{}l@{}} $- 5$ \end{tabular}\\ \hline \\ 
$1$ & $(3, 6, 1)$ & \begin{tabular}{@{}l@{}} $(0, 1, 0, 1, 1, 0, 2, 1, 1, 1)$ \end{tabular} & \begin{tabular}{@{}l@{}} 32 \end{tabular} & \begin{tabular}{@{}l@{}} $x_{1}x_{2}^{2}x_{3}^{2}x_{4}^{5}x_{5}^{5}x_{6}^{2}x_{8}^{5}x_{9}^{5}x_{10}^{5}$ \end{tabular} & \begin{tabular}{@{}l@{}} $2^2$ \end{tabular}\\ \hline \\ 
$1$ & $(1, 6, 3)$ & \begin{tabular}{@{}l@{}} $(0, 1, 0, 1, 1, 0, 2, 1, 1, 1)$ \end{tabular} & \begin{tabular}{@{}l@{}} 32 \end{tabular} & \begin{tabular}{@{}l@{}} $x_{1}x_{2}^{2}x_{3}^{2}x_{4}^{5}x_{5}^{5}x_{6}^{2}x_{8}^{5}x_{9}^{5}x_{10}^{5}$ \end{tabular} & \begin{tabular}{@{}l@{}} $2^2$ \end{tabular}\\ \hline \\ 
$1$ & $(4, 6, 0)$ & \begin{tabular}{@{}l@{}} $(0, 1, 1, 0, 1, 1, 1, 1, 0, 0)$ \end{tabular} & \begin{tabular}{@{}l@{}} 34 \end{tabular} & \begin{tabular}{@{}l@{}} $x_{1}x_{2}x_{3}^{3}x_{4}^{3}x_{5}^{5}x_{6}^{5}x_{7}^{5}x_{8}^{5}x_{9}^{3}x_{10}^{3}$ \end{tabular} & \begin{tabular}{@{}l@{}} $- 2 \cdot 3$ \end{tabular}\\ \hline \\ 
$1$ & $(0, 6, 4)$ & \begin{tabular}{@{}l@{}} $(2, 1, 1, 2, 1, 1, 1, 1, 2, 2)$ \end{tabular} & \begin{tabular}{@{}l@{}} 37 \end{tabular} & \begin{tabular}{@{}l@{}} $x_{1}^{2}x_{2}^{2}x_{3}^{4}x_{4}^{3}x_{5}^{5}x_{6}^{5}x_{7}^{5}x_{8}^{5}x_{9}^{3}x_{10}^{3}$ \end{tabular} & \begin{tabular}{@{}l@{}} $2^2$ \end{tabular}\\ \hline \\ 
$1$ & $(2, 7, 1)$ & \begin{tabular}{@{}l@{}} $(0, 1, 1, 0, 1, 1, 1, 1, 1, 2)$ \end{tabular} & \begin{tabular}{@{}l@{}} 35 \end{tabular} & \begin{tabular}{@{}l@{}} $x_{1}x_{2}x_{3}^{3}x_{4}x_{5}^{5}x_{6}^{6}x_{7}^{6}x_{8}^{6}x_{9}^{6}$ \end{tabular} & \begin{tabular}{@{}l@{}} $- 2$ \end{tabular}\\ \hline \\ 
$1$ & $(1, 7, 2)$ & \begin{tabular}{@{}l@{}} $(1, 0, 1, 2, 1, 1, 1, 1, 1, 2)$ \end{tabular} & \begin{tabular}{@{}l@{}} 40 \end{tabular} & \begin{tabular}{@{}l@{}} $x_{1}^{3}x_{3}^{5}x_{4}x_{5}^{6}x_{6}^{6}x_{7}^{6}x_{8}^{6}x_{9}^{6}x_{10}$ \end{tabular} & \begin{tabular}{@{}l@{}} $3^2$ \end{tabular}\\ \hline \\ 
$1$ & $(3, 7, 0)$ & \begin{tabular}{@{}l@{}} $(0, 1, 1, 0, 1, 1, 1, 1, 1, 0)$ \end{tabular} & \begin{tabular}{@{}l@{}} 36 \end{tabular} & \begin{tabular}{@{}l@{}} $x_{1}x_{2}x_{3}^{2}x_{4}^{2}x_{5}^{4}x_{6}^{6}x_{7}^{6}x_{8}^{6}x_{9}^{6}x_{10}^{2}$ \end{tabular} & \begin{tabular}{@{}l@{}} $1$ \end{tabular}\\ \hline \\ 
$1$ & $(0, 7, 3)$ & \begin{tabular}{@{}l@{}} $(2, 1, 1, 2, 1, 1, 1, 1, 1, 2)$ \end{tabular} & \begin{tabular}{@{}l@{}} 40 \end{tabular} & \begin{tabular}{@{}l@{}} $x_{1}^{2}x_{2}^{2}x_{3}^{3}x_{4}^{2}x_{5}^{5}x_{6}^{6}x_{7}^{6}x_{8}^{6}x_{9}^{6}x_{10}^{2}$ \end{tabular} & \begin{tabular}{@{}l@{}} $- 2^2$ \end{tabular}\\ \hline \\ 
$1$ & $(1, 8, 1)$ & \begin{tabular}{@{}l@{}} $(1, 0, 1, 1, 1, 1, 1, 1, 1, 2)$ \end{tabular} & \begin{tabular}{@{}l@{}} 43 \end{tabular} & \begin{tabular}{@{}l@{}} $x_{1}^{3}x_{3}^{3}x_{4}^{3}x_{5}^{6}x_{6}^{7}x_{7}^{7}x_{8}^{7}x_{9}^{7}$ \end{tabular} & \begin{tabular}{@{}l@{}} $1$ \end{tabular}\\ \hline \\ 
$1$ & $(2, 8, 0)$ & \begin{tabular}{@{}l@{}} $(1, 0, 1, 1, 1, 1, 1, 1, 1, 0)$ \end{tabular} & \begin{tabular}{@{}l@{}} 43 \end{tabular} & \begin{tabular}{@{}l@{}} $x_{1}^{2}x_{2}x_{3}^{3}x_{4}^{3}x_{5}^{5}x_{6}^{7}x_{7}^{7}x_{8}^{7}x_{9}^{7}x_{10}$ \end{tabular} & \begin{tabular}{@{}l@{}} $-1$ \end{tabular}\\ \hline \\ 
$1$ & $(0, 8, 2)$ & \begin{tabular}{@{}l@{}} $(1, 2, 1, 1, 1, 1, 1, 1, 1, 2)$ \end{tabular} & \begin{tabular}{@{}l@{}} 47 \end{tabular} & \begin{tabular}{@{}l@{}} $x_{1}^{3}x_{2}x_{3}^{4}x_{4}^{3}x_{5}^{7}x_{6}^{7}x_{7}^{7}x_{8}^{7}x_{9}^{7}x_{10}$ \end{tabular} & \begin{tabular}{@{}l@{}} $2^2$ \end{tabular}\\ \hline \\ 
$1$ & $(1, 9, 0)$ & \begin{tabular}{@{}l@{}} $(1, 0, 1, 1, 1, 1, 1, 1, 1, 1)$ \end{tabular} & \begin{tabular}{@{}l@{}} 51 \end{tabular} & \begin{tabular}{@{}l@{}} $x_{1}^{3}x_{3}^{2}x_{4}^{4}x_{5}^{5}x_{6}^{5}x_{7}^{8}x_{8}^{8}x_{9}^{8}x_{10}^{8}$ \end{tabular} & \begin{tabular}{@{}l@{}} $2$ \end{tabular}\\ \hline \\ 
$1$ & $(0, 9, 1)$ & \begin{tabular}{@{}l@{}} $(1, 2, 1, 1, 1, 1, 1, 1, 1, 1)$ \end{tabular} & \begin{tabular}{@{}l@{}} 54 \end{tabular} & \begin{tabular}{@{}l@{}} $x_{1}^{2}x_{2}^{2}x_{3}^{3}x_{4}^{3}x_{5}^{5}x_{6}^{6}x_{7}^{6}x_{8}^{9}x_{9}^{9}x_{10}^{9}$ \end{tabular} & \begin{tabular}{@{}l@{}} $2$ \end{tabular}\\ \hline \\ 
$1$ & $(0, 10, 0)$ & \begin{tabular}{@{}l@{}} $(1, 1, 1, 1, 1, 1, 1, 1, 1, 1)$ \end{tabular} & \begin{tabular}{@{}l@{}} 65 \end{tabular} & \begin{tabular}{@{}l@{}} $x_{1}^{3}x_{2}^{4}x_{3}^{4}x_{4}^{4}x_{5}^{6}x_{6}^{8}x_{7}^{9}x_{8}^{9}x_{9}^{9}x_{10}^{9}$ \end{tabular} & \begin{tabular}{@{}l@{}} $2 \cdot 5$ \end{tabular}\\ \hline \\ 
$2$ & $(3, 2, 5)$ & \begin{tabular}{@{}l@{}} $(0, 2, 2, 0, 2, 2, 2, 1, 0, 1)$ \end{tabular} & \begin{tabular}{@{}l@{}} 27 \end{tabular} & \begin{tabular}{@{}l@{}} $x_{1}x_{2}^{4}x_{3}^{4}x_{4}^{2}x_{5}^{4}x_{6}^{4}x_{7}^{4}x_{8}x_{9}^{2}x_{10}$ \\ $x_{1}^{2}x_{2}^{3}x_{3}^{4}x_{4}^{2}x_{5}^{4}x_{6}^{4}x_{7}^{4}x_{8}x_{9}^{2}x_{10}$ \end{tabular} & \begin{tabular}{@{}l@{}} $- 73$ \\ $- 233$ \end{tabular}\\ \hline \\ 
$2$ & $(2, 3, 5)$ & \begin{tabular}{@{}l@{}} $(0, 2, 0, 2, 1, 2, 2, 2, 1, 1)$ \end{tabular} & \begin{tabular}{@{}l@{}} 28 \end{tabular} & \begin{tabular}{@{}l@{}} $x_{1}x_{2}^{4}x_{3}x_{4}^{4}x_{5}^{2}x_{6}^{4}x_{7}^{4}x_{8}^{4}x_{9}^{2}x_{10}^{2}$ \end{tabular} & \begin{tabular}{@{}l@{}} $7$ \end{tabular}\\ \hline \\ 
$2$ & $(4, 1, 5)$ & \begin{tabular}{@{}l@{}} $(0, 2, 2, 0, 2, 2, 2, 0, 0, 1)$ \end{tabular} & \begin{tabular}{@{}l@{}} 31 \end{tabular} & \begin{tabular}{@{}l@{}} $x_{1}^{2}x_{2}^{4}x_{3}^{4}x_{4}^{3}x_{5}^{4}x_{6}^{4}x_{7}^{4}x_{8}^{3}x_{9}^{3}$ \\ $x_{1}^{3}x_{2}^{3}x_{3}^{4}x_{4}^{3}x_{5}^{4}x_{6}^{4}x_{7}^{4}x_{8}^{3}x_{9}^{3}$ \end{tabular} & \begin{tabular}{@{}l@{}} $- 5 \cdot 43$ \\ $- 3 \cdot 13 \cdot 17$ \end{tabular}\\ \hline \\ 
$2$ & $(1, 4, 5)$ & \begin{tabular}{@{}l@{}} $(2, 1, 1, 0, 1, 1, 2, 2, 2, 2)$ \end{tabular} & \begin{tabular}{@{}l@{}} 32 \end{tabular} & \begin{tabular}{@{}l@{}} $x_{1}^{4}x_{2}^{3}x_{3}^{3}x_{5}^{3}x_{6}^{3}x_{7}^{4}x_{8}^{4}x_{9}^{4}x_{10}^{4}$ \end{tabular} & \begin{tabular}{@{}l@{}} $- 7$ \end{tabular}\\ \hline \\ 
$2$ & $(5, 0, 5)$ & \begin{tabular}{@{}l@{}} $(0, 2, 0, 2, 0, 0, 0, 2, 2, 2)$ \end{tabular} & \begin{tabular}{@{}l@{}} 33 \end{tabular} & \begin{tabular}{@{}l@{}} $x_{1}^{2}x_{2}x_{3}^{4}x_{4}^{4}x_{5}^{4}x_{6}^{4}x_{7}^{4}x_{8}^{2}x_{9}^{4}x_{10}^{4}$ \end{tabular} & \begin{tabular}{@{}l@{}} $3 \cdot 5^2$ \end{tabular}\\ \hline \\ 
$2$ & $(2, 2, 6)$ & \begin{tabular}{@{}l@{}} $(0, 2, 0, 2, 2, 1, 1, 2, 2, 2)$ \end{tabular} & \begin{tabular}{@{}l@{}} 30 \end{tabular} & \begin{tabular}{@{}l@{}} $x_{1}x_{2}^{2}x_{3}x_{4}^{4}x_{5}^{5}x_{6}x_{7}x_{8}^{5}x_{9}^{5}x_{10}^{5}$ \end{tabular} & \begin{tabular}{@{}l@{}} $- 5$ \end{tabular}\\ \hline \\ 
$2$ & $(3, 1, 6)$ & \begin{tabular}{@{}l@{}} $(0, 2, 0, 2, 2, 0, 1, 2, 2, 2)$ \end{tabular} & \begin{tabular}{@{}l@{}} 32 \end{tabular} & \begin{tabular}{@{}l@{}} $x_{1}x_{2}^{2}x_{3}^{2}x_{4}^{5}x_{5}^{5}x_{6}^{2}x_{8}^{5}x_{9}^{5}x_{10}^{5}$ \end{tabular} & \begin{tabular}{@{}l@{}} $2^2$ \end{tabular}\\ \hline \\ 
$2$ & $(1, 3, 6)$ & \begin{tabular}{@{}l@{}} $(0, 2, 0, 2, 2, 0, 1, 2, 2, 2)$ \end{tabular} & \begin{tabular}{@{}l@{}} 32 \end{tabular} & \begin{tabular}{@{}l@{}} $x_{1}x_{2}^{2}x_{3}^{2}x_{4}^{5}x_{5}^{5}x_{6}^{2}x_{8}^{5}x_{9}^{5}x_{10}^{5}$ \end{tabular} & \begin{tabular}{@{}l@{}} $2^2$ \end{tabular}\\ \hline \\ 
$2$ & $(4, 0, 6)$ & \begin{tabular}{@{}l@{}} $(0, 2, 2, 0, 2, 2, 2, 2, 0, 0)$ \end{tabular} & \begin{tabular}{@{}l@{}} 34 \end{tabular} & \begin{tabular}{@{}l@{}} $x_{1}x_{2}x_{3}^{3}x_{4}^{3}x_{5}^{5}x_{6}^{5}x_{7}^{5}x_{8}^{5}x_{9}^{3}x_{10}^{3}$ \end{tabular} & \begin{tabular}{@{}l@{}} $- 2 \cdot 3$ \end{tabular}\\ \hline \\ 
$2$ & $(0, 4, 6)$ & \begin{tabular}{@{}l@{}} $(1, 2, 2, 1, 2, 2, 2, 2, 1, 1)$ \end{tabular} & \begin{tabular}{@{}l@{}} 37 \end{tabular} & \begin{tabular}{@{}l@{}} $x_{1}^{2}x_{2}^{2}x_{3}^{4}x_{4}^{3}x_{5}^{5}x_{6}^{5}x_{7}^{5}x_{8}^{5}x_{9}^{3}x_{10}^{3}$ \end{tabular} & \begin{tabular}{@{}l@{}} $2^2$ \end{tabular}\\ \hline \\ 
$2$ & $(2, 1, 7)$ & \begin{tabular}{@{}l@{}} $(0, 2, 2, 0, 2, 2, 2, 2, 2, 1)$ \end{tabular} & \begin{tabular}{@{}l@{}} 35 \end{tabular} & \begin{tabular}{@{}l@{}} $x_{1}x_{2}x_{3}^{3}x_{4}x_{5}^{5}x_{6}^{6}x_{7}^{6}x_{8}^{6}x_{9}^{6}$ \end{tabular} & \begin{tabular}{@{}l@{}} $- 2$ \end{tabular}\\ \hline \\ 
$2$ & $(1, 2, 7)$ & \begin{tabular}{@{}l@{}} $(2, 0, 2, 1, 2, 2, 2, 2, 2, 1)$ \end{tabular} & \begin{tabular}{@{}l@{}} 40 \end{tabular} & \begin{tabular}{@{}l@{}} $x_{1}^{3}x_{3}^{5}x_{4}x_{5}^{6}x_{6}^{6}x_{7}^{6}x_{8}^{6}x_{9}^{6}x_{10}$ \end{tabular} & \begin{tabular}{@{}l@{}} $3^2$ \end{tabular}\\ \hline \\ 
$2$ & $(3, 0, 7)$ & \begin{tabular}{@{}l@{}} $(0, 2, 2, 0, 2, 2, 2, 2, 2, 0)$ \end{tabular} & \begin{tabular}{@{}l@{}} 36 \end{tabular} & \begin{tabular}{@{}l@{}} $x_{1}x_{2}x_{3}^{2}x_{4}^{2}x_{5}^{4}x_{6}^{6}x_{7}^{6}x_{8}^{6}x_{9}^{6}x_{10}^{2}$ \end{tabular} & \begin{tabular}{@{}l@{}} $1$ \end{tabular}\\ \hline \\ 
$2$ & $(0, 3, 7)$ & \begin{tabular}{@{}l@{}} $(1, 2, 2, 1, 2, 2, 2, 2, 2, 1)$ \end{tabular} & \begin{tabular}{@{}l@{}} 40 \end{tabular} & \begin{tabular}{@{}l@{}} $x_{1}^{2}x_{2}^{2}x_{3}^{3}x_{4}^{2}x_{5}^{5}x_{6}^{6}x_{7}^{6}x_{8}^{6}x_{9}^{6}x_{10}^{2}$ \end{tabular} & \begin{tabular}{@{}l@{}} $- 2^2$ \end{tabular}\\ \hline \\ 
$2$ & $(1, 1, 8)$ & \begin{tabular}{@{}l@{}} $(2, 0, 2, 2, 2, 2, 2, 2, 2, 1)$ \end{tabular} & \begin{tabular}{@{}l@{}} 43 \end{tabular} & \begin{tabular}{@{}l@{}} $x_{1}^{3}x_{3}^{3}x_{4}^{3}x_{5}^{6}x_{6}^{7}x_{7}^{7}x_{8}^{7}x_{9}^{7}$ \end{tabular} & \begin{tabular}{@{}l@{}} $1$ \end{tabular}\\ \hline \\ 
$2$ & $(2, 0, 8)$ & \begin{tabular}{@{}l@{}} $(2, 0, 2, 2, 2, 2, 2, 2, 2, 0)$ \end{tabular} & \begin{tabular}{@{}l@{}} 43 \end{tabular} & \begin{tabular}{@{}l@{}} $x_{1}^{2}x_{2}x_{3}^{3}x_{4}^{3}x_{5}^{5}x_{6}^{7}x_{7}^{7}x_{8}^{7}x_{9}^{7}x_{10}$ \end{tabular} & \begin{tabular}{@{}l@{}} $-1$ \end{tabular}\\ \hline \\ 
$2$ & $(0, 2, 8)$ & \begin{tabular}{@{}l@{}} $(2, 1, 2, 2, 2, 2, 2, 2, 2, 1)$ \end{tabular} & \begin{tabular}{@{}l@{}} 47 \end{tabular} & \begin{tabular}{@{}l@{}} $x_{1}^{3}x_{2}x_{3}^{4}x_{4}^{3}x_{5}^{7}x_{6}^{7}x_{7}^{7}x_{8}^{7}x_{9}^{7}x_{10}$ \end{tabular} & \begin{tabular}{@{}l@{}} $2^2$ \end{tabular}\\ \hline \\ 
$2$ & $(1, 0, 9)$ & \begin{tabular}{@{}l@{}} $(2, 0, 2, 2, 2, 2, 2, 2, 2, 2)$ \end{tabular} & \begin{tabular}{@{}l@{}} 51 \end{tabular} & \begin{tabular}{@{}l@{}} $x_{1}^{3}x_{3}^{2}x_{4}^{4}x_{5}^{5}x_{6}^{5}x_{7}^{8}x_{8}^{8}x_{9}^{8}x_{10}^{8}$ \end{tabular} & \begin{tabular}{@{}l@{}} $2$ \end{tabular}\\ \hline \\ 
$2$ & $(0, 1, 9)$ & \begin{tabular}{@{}l@{}} $(2, 1, 2, 2, 2, 2, 2, 2, 2, 2)$ \end{tabular} & \begin{tabular}{@{}l@{}} 54 \end{tabular} & \begin{tabular}{@{}l@{}} $x_{1}^{2}x_{2}^{2}x_{3}^{3}x_{4}^{3}x_{5}^{5}x_{6}^{6}x_{7}^{6}x_{8}^{9}x_{9}^{9}x_{10}^{9}$ \end{tabular} & \begin{tabular}{@{}l@{}} $2$ \end{tabular}\\ \hline \\ 
$2$ & $(0, 0, 10)$ & \begin{tabular}{@{}l@{}} $(2, 2, 2, 2, 2, 2, 2, 2, 2, 2)$ \end{tabular} & \begin{tabular}{@{}l@{}} 65 \end{tabular} & \begin{tabular}{@{}l@{}} $x_{1}^{3}x_{2}^{4}x_{3}^{4}x_{4}^{4}x_{5}^{6}x_{6}^{8}x_{7}^{9}x_{8}^{9}x_{9}^{9}x_{10}^{9}$ \end{tabular} & \begin{tabular}{@{}l@{}} $2 \cdot 5$ \end{tabular}\\ \hline \\
\end{longtable}

\footnotesize
\begin{longtable}{llllll}
\caption{Monomials and their coefficients sufficient for the proof of Theorem~\ref{th:main2} in the case $t=6$ and~$G=D_{2p}$.}\label{tab:6_D}\\
\hline
$\bar{a}$ & $\bm{\lambda}$ & $\mathbf{a''}$ & deg & monomial/s & coefficient/s \\
\hline
\endfirsthead
\hline
$\bar{a}$ & $\bm{\lambda}$ & $\mathbf{a''}$ & deg & monomial/s & coefficient/s \\
\hline
\endhead
$0$ & $(5, 5)$ & \begin{tabular}{@{}l@{}} $(1, 0, 0, 0, 1, 1, 1, 0, 1, 0)$ \end{tabular} & \begin{tabular}{@{}l@{}} $34$ \end{tabular} & \begin{tabular}{@{}l@{}} $x_{2}^{2}x_{3}^{4}x_{4}^{4}x_{5}^{4}x_{6}^{4}x_{7}^{4}x_{8}^{4}x_{9}^{4}x_{10}^{4}$ \\ $x_{2}^{3}x_{3}^{3}x_{4}^{4}x_{5}^{4}x_{6}^{4}x_{7}^{4}x_{8}^{4}x_{9}^{4}x_{10}^{4}$ \end{tabular} & \begin{tabular}{@{}l@{}} $2^6 \cdot 5 \cdot 13$ \\ $2^4 \cdot 17 \cdot 23$ \end{tabular}\\ \hline \\ 
$0$ & $(6, 4)$ & \begin{tabular}{@{}l@{}} $(1, 0, 0, 0, 1, 0, 1, 0, 1, 0)$ \end{tabular} & \begin{tabular}{@{}l@{}} $38$ \end{tabular} & \begin{tabular}{@{}l@{}} $x_{2}^{4}x_{3}^{5}x_{4}^{5}x_{5}^{3}x_{6}^{5}x_{7}^{3}x_{8}^{5}x_{9}^{3}x_{10}^{5}$ \\ $x_{1}x_{2}^{3}x_{3}^{5}x_{4}^{5}x_{5}^{3}x_{6}^{5}x_{7}^{3}x_{8}^{5}x_{9}^{3}x_{10}^{5}$ \end{tabular} & \begin{tabular}{@{}l@{}} $- 3 \cdot 41543$ \\ $2^2 \cdot 3 \cdot 31 \cdot 617$ \end{tabular}\\ \hline \\ 
$0$ & $(7, 3)$ & \begin{tabular}{@{}l@{}} $(1, 0, 0, 0, 1, 0, 0, 0, 1, 0)$ \end{tabular} & \begin{tabular}{@{}l@{}} $36$ \end{tabular} & \begin{tabular}{@{}l@{}} $x_{3}^{2}x_{4}^{6}x_{5}^{2}x_{6}^{6}x_{7}^{6}x_{8}^{6}x_{9}^{2}x_{10}^{6}$ \end{tabular} & \begin{tabular}{@{}l@{}} $- 2^3 \cdot 3$ \end{tabular}\\ \hline \\ 
$0$ & $(8, 2)$ & \begin{tabular}{@{}l@{}} $(1, 0, 0, 0, 1, 0, 0, 0, 0, 0)$ \end{tabular} & \begin{tabular}{@{}l@{}} $45$ \end{tabular} & \begin{tabular}{@{}l@{}} $x_{3}^{2}x_{4}^{7}x_{5}x_{6}^{7}x_{7}^{7}x_{8}^{7}x_{9}^{7}x_{10}^{7}$ \end{tabular} & \begin{tabular}{@{}l@{}} $- 2^2$ \end{tabular}\\ \hline \\ 
$0$ & $(9, 1)$ & \begin{tabular}{@{}l@{}} $(0, 0, 0, 0, 1, 0, 0, 0, 0, 0)$ \end{tabular} & \begin{tabular}{@{}l@{}} $66$ \end{tabular} & \begin{tabular}{@{}l@{}} $x_{1}^{5}x_{2}^{7}x_{3}^{7}x_{4}^{8}x_{6}^{7}x_{7}^{8}x_{8}^{8}x_{9}^{8}x_{10}^{8}$ \end{tabular} & \begin{tabular}{@{}l@{}} $- 2^5 \cdot 3$ \end{tabular}\\ \hline \\ 
$1$ & $(5, 5)$ & \begin{tabular}{@{}l@{}} $(1, 0, 1, 1, 1, 0, 0, 0, 1, 0)$ \end{tabular} & \begin{tabular}{@{}l@{}} $39$ \end{tabular} & \begin{tabular}{@{}l@{}} $x_{1}^{3}x_{2}^{4}x_{3}^{4}x_{4}^{4}x_{5}^{4}x_{6}^{4}x_{7}^{4}x_{8}^{4}x_{9}^{4}x_{10}^{4}$ \\ $x_{1}^{4}x_{2}^{3}x_{3}^{4}x_{4}^{4}x_{5}^{4}x_{6}^{4}x_{7}^{4}x_{8}^{4}x_{9}^{4}x_{10}^{4}$ \end{tabular} & \begin{tabular}{@{}l@{}} $43 \cdot 97$ \\ $- 2 \cdot 3^3 \cdot 257$ \end{tabular}\\ \hline \\ 
$1$ & $(4, 6)$ & \begin{tabular}{@{}l@{}} $(0, 1, 1, 1, 1, 1, 0, 0, 1, 0)$ \end{tabular} & \begin{tabular}{@{}l@{}} $39$ \end{tabular} & \begin{tabular}{@{}l@{}} $x_{1}^{2}x_{2}^{3}x_{3}^{5}x_{4}^{5}x_{5}^{5}x_{6}^{5}x_{7}^{3}x_{8}^{3}x_{9}^{5}x_{10}^{3}$ \\ $x_{1}^{3}x_{2}^{2}x_{3}^{5}x_{4}^{5}x_{5}^{5}x_{6}^{5}x_{7}^{3}x_{8}^{3}x_{9}^{5}x_{10}^{3}$ \end{tabular} & \begin{tabular}{@{}l@{}} $- 11 \cdot 23 \cdot 61$ \\ $- 7^2 \cdot 953$ \end{tabular}\\ \hline \\ 
$1$ & $(3, 7)$ & \begin{tabular}{@{}l@{}} $(1, 0, 1, 1, 1, 1, 1, 0, 1, 0)$ \end{tabular} & \begin{tabular}{@{}l@{}} $43$ \end{tabular} & \begin{tabular}{@{}l@{}} $x_{1}^{3}x_{2}^{2}x_{3}^{4}x_{4}^{6}x_{5}^{6}x_{6}^{6}x_{7}^{6}x_{8}^{2}x_{9}^{6}x_{10}^{2}$ \\ $x_{1}^{4}x_{2}^{2}x_{3}^{3}x_{4}^{6}x_{5}^{6}x_{6}^{6}x_{7}^{6}x_{8}^{2}x_{9}^{6}x_{10}^{2}$ \end{tabular} & \begin{tabular}{@{}l@{}} $- 5 \cdot 23 \cdot 61$ \\ $2 \cdot 32117$ \end{tabular}\\ \hline \\ 
$1$ & $(2, 8)$ & \begin{tabular}{@{}l@{}} $(1, 0, 1, 1, 1, 1, 1, 1, 1, 0)$ \end{tabular} & \begin{tabular}{@{}l@{}} $48$ \end{tabular} & \begin{tabular}{@{}l@{}} $x_{1}^{4}x_{2}x_{3}^{3}x_{4}^{4}x_{5}^{7}x_{6}^{7}x_{7}^{7}x_{8}^{7}x_{9}^{7}x_{10}$ \\ $x_{1}^{4}x_{2}x_{3}^{3}x_{4}^{5}x_{5}^{6}x_{6}^{7}x_{7}^{7}x_{8}^{7}x_{9}^{7}x_{10}$ \end{tabular} & \begin{tabular}{@{}l@{}} $- 2^2 \cdot 5 \cdot 61$ \\ $2^4 \cdot 3^2 \cdot 29$ \end{tabular}\\ \hline \\ 
$1$ & $(1, 9)$ & \begin{tabular}{@{}l@{}} $(1, 0, 1, 1, 1, 1, 1, 1, 1, 1)$ \end{tabular} & \begin{tabular}{@{}l@{}} $55$ \end{tabular} & \begin{tabular}{@{}l@{}} $x_{1}^{5}x_{3}^{3}x_{4}^{4}x_{5}^{4}x_{6}^{7}x_{7}^{8}x_{8}^{8}x_{9}^{8}x_{10}^{8}$ \\ $x_{1}^{5}x_{3}^{3}x_{4}^{4}x_{5}^{5}x_{6}^{6}x_{7}^{8}x_{8}^{8}x_{9}^{8}x_{10}^{8}$ \end{tabular} & \begin{tabular}{@{}l@{}} $2 \cdot 3 \cdot 61$ \\ $- 2^4 \cdot 83$ \end{tabular}\\ \hline \\ 
$1$ & $(0, 10)$ & \begin{tabular}{@{}l@{}} $(1, 1, 1, 1, 1, 1, 1, 1, 1, 1)$ \end{tabular} & \begin{tabular}{@{}l@{}} $66$ \end{tabular} & \begin{tabular}{@{}l@{}} $x_{1}^{4}x_{2}^{4}x_{3}^{5}x_{4}^{8}x_{6}^{9}x_{7}^{9}x_{8}^{9}x_{9}^{9}x_{10}^{9}$ \\ $x_{1}^{4}x_{2}^{4}x_{3}^{5}x_{4}^{9}x_{6}^{8}x_{7}^{9}x_{8}^{9}x_{9}^{9}x_{10}^{9}$ \end{tabular} & \begin{tabular}{@{}l@{}} $- 2 \cdot 457$ \\ $- 2^8 \cdot 3^2$ \end{tabular}\\ \hline
\end{longtable}
\footnotesize
\begin{longtable}{llllll}
\caption{Monomials and their coefficients sufficient for the proof of Theorem~\ref{th:main2} in the case $t=6$ and~$G=G_{3p}$.}\label{tab:6_G}\\
\hline
$\bar{a}$ & $\bm{\lambda}$ & $\mathbf{a''}$ & deg & monomial/s & coefficient/s \\
\hline
\endfirsthead
\hline
$\bar{a}$ & $\bm{\lambda}$ & $\mathbf{a''}$ & deg & monomial/s & coefficient/s \\
\hline
\endhead
$0$ & $(5, 3, 2)$ & \begin{tabular}{@{}l@{}} $(1, 0, 0, 1, 0, 0, 0, 2, 1, 2)$ \end{tabular} & \begin{tabular}{@{}l@{}} 25 \end{tabular} & \begin{tabular}{@{}l@{}} $x_{1}^{2}x_{2}^{4}x_{3}^{4}x_{4}^{2}x_{5}^{2}x_{6}^{4}x_{7}^{4}x_{9}^{2}x_{10}$ \end{tabular} & \begin{tabular}{@{}l@{}} $-r - 1$ \end{tabular}\\ \hline \\ 
$0$ & $(5, 4, 1)$ & \begin{tabular}{@{}l@{}} $(1, 0, 0, 1, 0, 0, 0, 1, 1, 2)$ \end{tabular} & \begin{tabular}{@{}l@{}} 22 \end{tabular} & \begin{tabular}{@{}l@{}} $x_{1}^{3}x_{2}^{4}x_{3}^{3}x_{4}^{3}x_{5}x_{6}^{2}x_{7}^{4}x_{8}x_{9}$ \end{tabular} & \begin{tabular}{@{}l@{}} $1$ \end{tabular}\\ \hline \\ 
$0$ & $(5, 5, 0)$ & \begin{tabular}{@{}l@{}} $(0, 1, 0, 1, 0, 0, 0, 1, 1, 1)$ \end{tabular} & \begin{tabular}{@{}l@{}} 33 \end{tabular} & \begin{tabular}{@{}l@{}} $x_{1}^{2}x_{2}^{2}x_{3}^{2}x_{4}^{3}x_{5}^{4}x_{6}^{4}x_{7}^{4}x_{8}^{4}x_{9}^{4}x_{10}^{4}$ \end{tabular} & \begin{tabular}{@{}l@{}} $124r + 294$ \end{tabular}\\ \hline \\ 
$0$ & $(6, 2, 2)$ & \begin{tabular}{@{}l@{}} $(1, 0, 0, 1, 0, 0, 0, 0, 2, 2)$ \end{tabular} & \begin{tabular}{@{}l@{}} 25 \end{tabular} & \begin{tabular}{@{}l@{}} $x_{1}x_{2}^{5}x_{3}^{2}x_{4}x_{5}^{3}x_{6}^{2}x_{7}^{5}x_{8}^{5}x_{10}$ \end{tabular} & \begin{tabular}{@{}l@{}} $1$ \end{tabular}\\ \hline \\ 
$0$ & $(6, 3, 1)$ & \begin{tabular}{@{}l@{}} $(1, 0, 0, 1, 0, 0, 0, 0, 1, 2)$ \end{tabular} & \begin{tabular}{@{}l@{}} 30 \end{tabular} & \begin{tabular}{@{}l@{}} $x_{1}^{2}x_{2}^{3}x_{3}^{4}x_{4}x_{5}^{5}x_{6}^{5}x_{7}^{4}x_{8}^{4}x_{9}^{2}$ \end{tabular} & \begin{tabular}{@{}l@{}} $1$ \end{tabular}\\ \hline \\ 
$0$ & $(6, 4, 0)$ & \begin{tabular}{@{}l@{}} $(1, 0, 0, 1, 0, 0, 0, 0, 1, 1)$ \end{tabular} & \begin{tabular}{@{}l@{}} 28 \end{tabular} & \begin{tabular}{@{}l@{}} $x_{2}^{2}x_{3}^{3}x_{4}x_{5}^{4}x_{6}^{4}x_{7}^{5}x_{8}^{4}x_{9}^{2}x_{10}^{3}$ \end{tabular} & \begin{tabular}{@{}l@{}} $2$ \end{tabular}\\ \hline \\ 
$0$ & $(7, 2, 1)$ & \begin{tabular}{@{}l@{}} $(1, 0, 0, 1, 0, 0, 0, 0, 0, 2)$ \end{tabular} & \begin{tabular}{@{}l@{}} 33 \end{tabular} & \begin{tabular}{@{}l@{}} $x_{3}^{3}x_{4}x_{5}^{6}x_{6}^{5}x_{7}^{6}x_{8}^{6}x_{9}^{6}$ \end{tabular} & \begin{tabular}{@{}l@{}} $2$ \end{tabular}\\ \hline \\ 
$0$ & $(7, 3, 0)$ & \begin{tabular}{@{}l@{}} $(1, 0, 0, 1, 0, 0, 0, 0, 0, 1)$ \end{tabular} & \begin{tabular}{@{}l@{}} 35 \end{tabular} & \begin{tabular}{@{}l@{}} $x_{3}^{3}x_{4}x_{5}^{6}x_{6}^{6}x_{7}^{6}x_{8}^{5}x_{9}^{6}x_{10}^{2}$ \end{tabular} & \begin{tabular}{@{}l@{}} $2$ \end{tabular}\\ \hline \\ 
$0$ & $(8, 1, 1)$ & \begin{tabular}{@{}l@{}} $(1, 0, 0, 0, 2, 0, 0, 0, 0, 0)$ \end{tabular} & \begin{tabular}{@{}l@{}} 44 \end{tabular} & \begin{tabular}{@{}l@{}} $x_{2}^{3}x_{3}^{4}x_{4}^{2}x_{6}^{7}x_{7}^{7}x_{8}^{7}x_{9}^{7}x_{10}^{7}$ \end{tabular} & \begin{tabular}{@{}l@{}} $12r + 12$ \end{tabular}\\ \hline \\ 
$0$ & $(8, 2, 0)$ & \begin{tabular}{@{}l@{}} $(1, 0, 0, 0, 1, 0, 0, 0, 0, 0)$ \end{tabular} & \begin{tabular}{@{}l@{}} 42 \end{tabular} & \begin{tabular}{@{}l@{}} $x_{2}^{3}x_{3}^{4}x_{4}^{4}x_{5}x_{6}^{5}x_{7}^{5}x_{8}^{7}x_{9}^{7}x_{10}^{6}$ \end{tabular} & \begin{tabular}{@{}l@{}} $5$ \end{tabular}\\ \hline \\ 
$0$ & $(9, 1, 0)$ & \begin{tabular}{@{}l@{}} $(0, 0, 0, 1, 0, 0, 0, 0, 0, 0)$ \end{tabular} & \begin{tabular}{@{}l@{}} 66 \end{tabular} & \begin{tabular}{@{}l@{}} $x_{1}^{5}x_{2}^{7}x_{3}^{7}x_{5}^{8}x_{6}^{8}x_{7}^{7}x_{8}^{8}x_{9}^{8}x_{10}^{8}$ \end{tabular} & \begin{tabular}{@{}l@{}} $2^2 \cdot 5^2$ \end{tabular}\\ \hline \\ 
$1$ & $(3, 5, 2)$ & \begin{tabular}{@{}l@{}} $(0, 1, 1, 0, 1, 1, 1, 2, 0, 2)$ \end{tabular} & \begin{tabular}{@{}l@{}} 27 \end{tabular} & \begin{tabular}{@{}l@{}} $x_{1}x_{2}^{4}x_{3}^{4}x_{4}^{2}x_{5}^{4}x_{6}^{4}x_{7}^{4}x_{8}x_{9}^{2}x_{10}$ \end{tabular} & \begin{tabular}{@{}l@{}} $-3055r - 1872$ \end{tabular}\\ \hline \\ 
$1$ & $(2, 5, 3)$ & \begin{tabular}{@{}l@{}} $(0, 1, 1, 0, 1, 1, 1, 2, 2, 2)$ \end{tabular} & \begin{tabular}{@{}l@{}} 28 \end{tabular} & \begin{tabular}{@{}l@{}} $x_{1}x_{2}^{4}x_{3}^{4}x_{4}x_{5}^{4}x_{6}^{4}x_{7}^{4}x_{8}^{2}x_{9}^{2}x_{10}^{2}$ \end{tabular} & \begin{tabular}{@{}l@{}} $3993r + 2619$ \end{tabular}\\ \hline \\ 
$1$ & $(4, 5, 1)$ & \begin{tabular}{@{}l@{}} $(0, 1, 1, 0, 1, 1, 1, 0, 0, 2)$ \end{tabular} & \begin{tabular}{@{}l@{}} 31 \end{tabular} & \begin{tabular}{@{}l@{}} $x_{1}^{2}x_{2}^{4}x_{3}^{4}x_{4}^{3}x_{5}^{4}x_{6}^{4}x_{7}^{4}x_{8}^{3}x_{9}^{3}$ \end{tabular} & \begin{tabular}{@{}l@{}} $3130r + 21$ \end{tabular}\\ \hline \\ 
$1$ & $(1, 5, 4)$ & \begin{tabular}{@{}l@{}} $(1, 2, 2, 1, 2, 2, 1, 1, 1, 0)$ \end{tabular} & \begin{tabular}{@{}l@{}} 32 \end{tabular} & \begin{tabular}{@{}l@{}} $x_{1}^{4}x_{2}^{3}x_{3}^{3}x_{4}^{4}x_{5}^{3}x_{6}^{3}x_{7}^{4}x_{8}^{4}x_{9}^{4}$ \end{tabular} & \begin{tabular}{@{}l@{}} $-4043r - 1625$ \end{tabular}\\ \hline \\ 
$1$ & $(0, 5, 5)$ & \begin{tabular}{@{}l@{}} $(1, 2, 1, 2, 1, 1, 1, 2, 2, 2)$ \end{tabular} & \begin{tabular}{@{}l@{}} 38 \end{tabular} & \begin{tabular}{@{}l@{}} $x_{1}^{3}x_{2}^{3}x_{3}^{4}x_{4}^{4}x_{5}^{4}x_{6}^{4}x_{7}^{4}x_{8}^{4}x_{9}^{4}x_{10}^{4}$ \end{tabular} & \begin{tabular}{@{}l@{}} $7128r - 12852$ \end{tabular}\\ \hline \\ 
$1$ & $(2, 6, 2)$ & \begin{tabular}{@{}l@{}} $(0, 1, 0, 1, 1, 2, 2, 1, 1, 1)$ \end{tabular} & \begin{tabular}{@{}l@{}} 30 \end{tabular} & \begin{tabular}{@{}l@{}} $x_{1}x_{2}^{3}x_{3}x_{4}^{3}x_{5}^{5}x_{6}x_{7}x_{8}^{5}x_{9}^{5}x_{10}^{5}$ \end{tabular} & \begin{tabular}{@{}l@{}} $-424r - 595$ \end{tabular}\\ \hline \\ 
$1$ & $(3, 6, 1)$ & \begin{tabular}{@{}l@{}} $(0, 1, 0, 1, 1, 0, 2, 1, 1, 1)$ \end{tabular} & \begin{tabular}{@{}l@{}} 36 \end{tabular} & \begin{tabular}{@{}l@{}} $x_{1}^{2}x_{2}^{5}x_{3}^{2}x_{4}^{5}x_{5}^{5}x_{6}^{2}x_{8}^{5}x_{9}^{5}x_{10}^{5}$ \end{tabular} & \begin{tabular}{@{}l@{}} $-566r + 1653$ \end{tabular}\\ \hline \\ 
$1$ & $(1, 6, 3)$ & \begin{tabular}{@{}l@{}} $(1, 0, 1, 1, 2, 2, 2, 1, 1, 1)$ \end{tabular} & \begin{tabular}{@{}l@{}} 34 \end{tabular} & \begin{tabular}{@{}l@{}} $x_{1}^{4}x_{3}^{4}x_{4}^{5}x_{5}^{2}x_{6}^{2}x_{7}^{2}x_{8}^{5}x_{9}^{5}x_{10}^{5}$ \end{tabular} & \begin{tabular}{@{}l@{}} $61364r + 25951$ \end{tabular}\\ \hline \\ 
$1$ & $(4, 6, 0)$ & \begin{tabular}{@{}l@{}} $(0, 1, 1, 0, 1, 1, 1, 1, 0, 0)$ \end{tabular} & \begin{tabular}{@{}l@{}} 34 \end{tabular} & \begin{tabular}{@{}l@{}} $x_{1}^{3}x_{2}^{3}x_{3}^{3}x_{4}^{3}x_{5}^{3}x_{6}^{3}x_{7}^{5}x_{8}^{5}x_{9}^{3}x_{10}^{3}$ \end{tabular} & \begin{tabular}{@{}l@{}} $-1272r + 653$ \end{tabular}\\ \hline \\ 
$1$ & $(0, 6, 4)$ & \begin{tabular}{@{}l@{}} $(2, 1, 1, 2, 1, 1, 1, 1, 2, 2)$ \end{tabular} & \begin{tabular}{@{}l@{}} 37 \end{tabular} & \begin{tabular}{@{}l@{}} $x_{1}^{3}x_{2}^{3}x_{3}^{3}x_{4}^{3}x_{5}^{4}x_{6}^{5}x_{7}^{5}x_{8}^{5}x_{9}^{3}x_{10}^{3}$ \end{tabular} & \begin{tabular}{@{}l@{}} $-5321r - 5949$ \end{tabular}\\ \hline \\ 
$1$ & $(2, 7, 1)$ & \begin{tabular}{@{}l@{}} $(0, 1, 1, 0, 1, 1, 1, 1, 1, 2)$ \end{tabular} & \begin{tabular}{@{}l@{}} 35 \end{tabular} & \begin{tabular}{@{}l@{}} $x_{1}x_{2}^{3}x_{3}^{3}x_{4}x_{5}^{3}x_{6}^{6}x_{7}^{6}x_{8}^{6}x_{9}^{6}$ \end{tabular} & \begin{tabular}{@{}l@{}} $-291r + 1$ \end{tabular}\\ \hline \\ 
$1$ & $(1, 7, 2)$ & \begin{tabular}{@{}l@{}} $(1, 0, 1, 2, 1, 1, 1, 1, 1, 2)$ \end{tabular} & \begin{tabular}{@{}l@{}} 40 \end{tabular} & \begin{tabular}{@{}l@{}} $x_{1}^{4}x_{3}^{4}x_{4}x_{5}^{6}x_{6}^{6}x_{7}^{6}x_{8}^{6}x_{9}^{6}x_{10}$ \end{tabular} & \begin{tabular}{@{}l@{}} $11711r + 1287$ \end{tabular}\\ \hline \\ 
$1$ & $(3, 7, 0)$ & \begin{tabular}{@{}l@{}} $(0, 1, 1, 0, 1, 1, 1, 1, 1, 0)$ \end{tabular} & \begin{tabular}{@{}l@{}} 36 \end{tabular} & \begin{tabular}{@{}l@{}} $x_{1}^{2}x_{2}^{3}x_{3}^{3}x_{4}^{2}x_{5}^{3}x_{6}^{3}x_{7}^{6}x_{8}^{6}x_{9}^{6}x_{10}^{2}$ \end{tabular} & \begin{tabular}{@{}l@{}} $5839r + 1451$ \end{tabular}\\ \hline \\ 
$1$ & $(0, 7, 3)$ & \begin{tabular}{@{}l@{}} $(2, 1, 1, 2, 1, 1, 1, 1, 1, 2)$ \end{tabular} & \begin{tabular}{@{}l@{}} 40 \end{tabular} & \begin{tabular}{@{}l@{}} $x_{1}^{2}x_{2}^{3}x_{3}^{3}x_{4}^{2}x_{5}^{4}x_{6}^{6}x_{7}^{6}x_{8}^{6}x_{9}^{6}x_{10}^{2}$ \end{tabular} & \begin{tabular}{@{}l@{}} $-1365r - 1577$ \end{tabular}\\ \hline \\ 
$1$ & $(1, 8, 1)$ & \begin{tabular}{@{}l@{}} $(1, 0, 1, 1, 1, 1, 1, 1, 1, 2)$ \end{tabular} & \begin{tabular}{@{}l@{}} 43 \end{tabular} & \begin{tabular}{@{}l@{}} $x_{1}^{4}x_{3}^{4}x_{4}^{4}x_{5}^{4}x_{6}^{6}x_{7}^{7}x_{8}^{7}x_{9}^{7}$ \end{tabular} & \begin{tabular}{@{}l@{}} $239r - 1811$ \end{tabular}\\ \hline \\ 
$1$ & $(2, 8, 0)$ & \begin{tabular}{@{}l@{}} $(1, 0, 1, 1, 1, 1, 1, 1, 1, 0)$ \end{tabular} & \begin{tabular}{@{}l@{}} 43 \end{tabular} & \begin{tabular}{@{}l@{}} $x_{1}^{4}x_{2}x_{3}^{4}x_{4}^{4}x_{5}^{4}x_{6}^{4}x_{7}^{7}x_{8}^{7}x_{9}^{7}x_{10}$ \end{tabular} & \begin{tabular}{@{}l@{}} $-10919r - 9004$ \end{tabular}\\ \hline \\ 
$1$ & $(0, 8, 2)$ & \begin{tabular}{@{}l@{}} $(1, 2, 1, 1, 1, 1, 1, 1, 1, 2)$ \end{tabular} & \begin{tabular}{@{}l@{}} 47 \end{tabular} & \begin{tabular}{@{}l@{}} $x_{1}^{4}x_{2}x_{3}^{4}x_{4}^{4}x_{5}^{5}x_{6}^{7}x_{7}^{7}x_{8}^{7}x_{9}^{7}x_{10}$ \end{tabular} & \begin{tabular}{@{}l@{}} $5278r + 39848$ \end{tabular}\\ \hline \\ 
$1$ & $(1, 9, 0)$ & \begin{tabular}{@{}l@{}} $(1, 0, 1, 1, 1, 1, 1, 1, 1, 1)$ \end{tabular} & \begin{tabular}{@{}l@{}} 51 \end{tabular} & \begin{tabular}{@{}l@{}} $x_{1}^{4}x_{3}^{4}x_{4}^{4}x_{5}^{5}x_{6}^{5}x_{7}^{5}x_{8}^{8}x_{9}^{8}x_{10}^{8}$ \end{tabular} & \begin{tabular}{@{}l@{}} $-12212r - 8422$ \end{tabular}\\ \hline \\ 
$1$ & $(0, 9, 1)$ & \begin{tabular}{@{}l@{}} $(1, 2, 1, 1, 1, 1, 1, 1, 1, 1)$ \end{tabular} & \begin{tabular}{@{}l@{}} 54 \end{tabular} & \begin{tabular}{@{}l@{}} $x_{1}^{4}x_{3}^{4}x_{4}^{4}x_{5}^{6}x_{6}^{6}x_{7}^{6}x_{8}^{8}x_{9}^{8}x_{10}^{8}$ \end{tabular} & \begin{tabular}{@{}l@{}} $-82561r - 26061$ \end{tabular}\\ \hline \\ 
$1$ & $(0, 10, 0)$ & \begin{tabular}{@{}l@{}} $(1, 1, 1, 1, 1, 1, 1, 1, 1, 1)$ \end{tabular} & \begin{tabular}{@{}l@{}} 65 \end{tabular} & \begin{tabular}{@{}l@{}} $x_{1}^{2}x_{2}^{4}x_{3}^{4}x_{4}^{4}x_{5}^{6}x_{6}^{9}x_{7}^{9}x_{8}^{9}x_{9}^{9}x_{10}^{9}$ \end{tabular} & \begin{tabular}{@{}l@{}} $7028r + 8686$ \end{tabular}\\ \hline \\ 
$2$ & $(3, 2, 5)$ & \begin{tabular}{@{}l@{}} $(0, 2, 2, 0, 2, 2, 2, 1, 0, 1)$ \end{tabular} & \begin{tabular}{@{}l@{}} 27 \end{tabular} & \begin{tabular}{@{}l@{}} $x_{1}x_{2}^{4}x_{3}^{4}x_{4}^{2}x_{5}^{4}x_{6}^{4}x_{7}^{4}x_{8}x_{9}^{2}x_{10}$ \end{tabular} & \begin{tabular}{@{}l@{}} $3055r + 1183$ \end{tabular}\\ \hline \\ 
$2$ & $(2, 3, 5)$ & \begin{tabular}{@{}l@{}} $(0, 2, 2, 0, 2, 2, 2, 1, 1, 1)$ \end{tabular} & \begin{tabular}{@{}l@{}} 28 \end{tabular} & \begin{tabular}{@{}l@{}} $x_{1}x_{2}^{4}x_{3}^{4}x_{4}x_{5}^{4}x_{6}^{4}x_{7}^{4}x_{8}^{2}x_{9}^{2}x_{10}^{2}$ \end{tabular} & \begin{tabular}{@{}l@{}} $-3993r - 1374$ \end{tabular}\\ \hline \\ 
$2$ & $(4, 1, 5)$ & \begin{tabular}{@{}l@{}} $(0, 2, 2, 0, 2, 2, 2, 0, 0, 1)$ \end{tabular} & \begin{tabular}{@{}l@{}} 31 \end{tabular} & \begin{tabular}{@{}l@{}} $x_{1}^{2}x_{2}^{4}x_{3}^{4}x_{4}^{3}x_{5}^{4}x_{6}^{4}x_{7}^{4}x_{8}^{3}x_{9}^{3}$ \end{tabular} & \begin{tabular}{@{}l@{}} $-3130r - 3109$ \end{tabular}\\ \hline \\ 
$2$ & $(1, 4, 5)$ & \begin{tabular}{@{}l@{}} $(2, 1, 1, 2, 1, 1, 2, 2, 2, 0)$ \end{tabular} & \begin{tabular}{@{}l@{}} 32 \end{tabular} & \begin{tabular}{@{}l@{}} $x_{1}^{4}x_{2}^{3}x_{3}^{3}x_{4}^{4}x_{5}^{3}x_{6}^{3}x_{7}^{4}x_{8}^{4}x_{9}^{4}$ \end{tabular} & \begin{tabular}{@{}l@{}} $4043r + 2418$ \end{tabular}\\ \hline \\ 
$2$ & $(5, 0, 5)$ & \begin{tabular}{@{}l@{}} $(0, 2, 0, 2, 0, 0, 0, 2, 2, 2)$ \end{tabular} & \begin{tabular}{@{}l@{}} 33 \end{tabular} & \begin{tabular}{@{}l@{}} $x_{1}^{2}x_{2}^{2}x_{3}^{2}x_{4}^{3}x_{5}^{4}x_{6}^{4}x_{7}^{4}x_{8}^{4}x_{9}^{4}x_{10}^{4}$ \end{tabular} & \begin{tabular}{@{}l@{}} $-124r + 170$ \end{tabular}\\ \hline \\ 
$2$ & $(2, 2, 6)$ & \begin{tabular}{@{}l@{}} $(0, 2, 0, 2, 2, 1, 1, 2, 2, 2)$ \end{tabular} & \begin{tabular}{@{}l@{}} 30 \end{tabular} & \begin{tabular}{@{}l@{}} $x_{1}x_{2}^{3}x_{3}x_{4}^{3}x_{5}^{5}x_{6}x_{7}x_{8}^{5}x_{9}^{5}x_{10}^{5}$ \end{tabular} & \begin{tabular}{@{}l@{}} $424r - 171$ \end{tabular}\\ \hline \\ 
$2$ & $(3, 1, 6)$ & \begin{tabular}{@{}l@{}} $(0, 2, 0, 2, 2, 0, 1, 2, 2, 2)$ \end{tabular} & \begin{tabular}{@{}l@{}} 36 \end{tabular} & \begin{tabular}{@{}l@{}} $x_{1}^{2}x_{2}^{5}x_{3}^{2}x_{4}^{5}x_{5}^{5}x_{6}^{2}x_{8}^{5}x_{9}^{5}x_{10}^{5}$ \end{tabular} & \begin{tabular}{@{}l@{}} $566r + 2219$ \end{tabular}\\ \hline \\ 
$2$ & $(1, 3, 6)$ & \begin{tabular}{@{}l@{}} $(2, 0, 2, 2, 1, 1, 1, 2, 2, 2)$ \end{tabular} & \begin{tabular}{@{}l@{}} 34 \end{tabular} & \begin{tabular}{@{}l@{}} $x_{1}^{4}x_{3}^{4}x_{4}^{5}x_{5}^{2}x_{6}^{2}x_{7}^{2}x_{8}^{5}x_{9}^{5}x_{10}^{5}$ \end{tabular} & \begin{tabular}{@{}l@{}} $-61364r - 35413$ \end{tabular}\\ \hline \\ 
$2$ & $(4, 0, 6)$ & \begin{tabular}{@{}l@{}} $(0, 2, 2, 0, 2, 2, 2, 2, 0, 0)$ \end{tabular} & \begin{tabular}{@{}l@{}} 34 \end{tabular} & \begin{tabular}{@{}l@{}} $x_{1}^{3}x_{2}^{3}x_{3}^{3}x_{4}^{3}x_{5}^{3}x_{6}^{3}x_{7}^{5}x_{8}^{5}x_{9}^{3}x_{10}^{3}$ \end{tabular} & \begin{tabular}{@{}l@{}} $1272r + 1925$ \end{tabular}\\ \hline \\ 
$2$ & $(0, 4, 6)$ & \begin{tabular}{@{}l@{}} $(1, 2, 2, 1, 2, 2, 2, 2, 1, 1)$ \end{tabular} & \begin{tabular}{@{}l@{}} 37 \end{tabular} & \begin{tabular}{@{}l@{}} $x_{1}^{3}x_{2}^{3}x_{3}^{3}x_{4}^{3}x_{5}^{4}x_{6}^{5}x_{7}^{5}x_{8}^{5}x_{9}^{3}x_{10}^{3}$ \end{tabular} & \begin{tabular}{@{}l@{}} $5321r - 628$ \end{tabular}\\ \hline \\ 
$2$ & $(2, 1, 7)$ & \begin{tabular}{@{}l@{}} $(0, 2, 2, 0, 2, 2, 2, 2, 2, 1)$ \end{tabular} & \begin{tabular}{@{}l@{}} 35 \end{tabular} & \begin{tabular}{@{}l@{}} $x_{1}x_{2}^{3}x_{3}^{3}x_{4}x_{5}^{3}x_{6}^{6}x_{7}^{6}x_{8}^{6}x_{9}^{6}$ \end{tabular} & \begin{tabular}{@{}l@{}} $291r + 29$ \end{tabular}\\ \hline \\ 
$2$ & $(1, 2, 7)$ & \begin{tabular}{@{}l@{}} $(2, 0, 2, 1, 2, 2, 2, 2, 2, 1)$ \end{tabular} & \begin{tabular}{@{}l@{}} 40 \end{tabular} & \begin{tabular}{@{}l@{}} $x_{1}^{4}x_{3}^{4}x_{4}x_{5}^{6}x_{6}^{6}x_{7}^{6}x_{8}^{6}x_{9}^{6}x_{10}$ \end{tabular} & \begin{tabular}{@{}l@{}} $-11711r - 10424$ \end{tabular}\\ \hline \\ 
$2$ & $(3, 0, 7)$ & \begin{tabular}{@{}l@{}} $(0, 2, 2, 0, 2, 2, 2, 2, 2, 0)$ \end{tabular} & \begin{tabular}{@{}l@{}} 36 \end{tabular} & \begin{tabular}{@{}l@{}} $x_{1}^{2}x_{2}^{3}x_{3}^{3}x_{4}^{2}x_{5}^{3}x_{6}^{3}x_{7}^{6}x_{8}^{6}x_{9}^{6}x_{10}^{2}$ \end{tabular} & \begin{tabular}{@{}l@{}} $-5839r - 4388$ \end{tabular}\\ \hline \\ 
$2$ & $(0, 3, 7)$ & \begin{tabular}{@{}l@{}} $(1, 2, 2, 1, 2, 2, 2, 2, 2, 1)$ \end{tabular} & \begin{tabular}{@{}l@{}} 40 \end{tabular} & \begin{tabular}{@{}l@{}} $x_{1}^{2}x_{2}^{3}x_{3}^{3}x_{4}^{2}x_{5}^{4}x_{6}^{6}x_{7}^{6}x_{8}^{6}x_{9}^{6}x_{10}^{2}$ \end{tabular} & \begin{tabular}{@{}l@{}} $1365r - 212$ \end{tabular}\\ \hline \\ 
$2$ & $(1, 1, 8)$ & \begin{tabular}{@{}l@{}} $(2, 0, 2, 2, 2, 2, 2, 2, 2, 1)$ \end{tabular} & \begin{tabular}{@{}l@{}} 43 \end{tabular} & \begin{tabular}{@{}l@{}} $x_{1}^{4}x_{3}^{4}x_{4}^{4}x_{5}^{4}x_{6}^{6}x_{7}^{7}x_{8}^{7}x_{9}^{7}$ \end{tabular} & \begin{tabular}{@{}l@{}} $-239r - 2050$ \end{tabular}\\ \hline \\ 
$2$ & $(2, 0, 8)$ & \begin{tabular}{@{}l@{}} $(2, 0, 2, 2, 2, 2, 2, 2, 2, 0)$ \end{tabular} & \begin{tabular}{@{}l@{}} 43 \end{tabular} & \begin{tabular}{@{}l@{}} $x_{1}^{4}x_{2}x_{3}^{4}x_{4}^{4}x_{5}^{4}x_{6}^{4}x_{7}^{7}x_{8}^{7}x_{9}^{7}x_{10}$ \end{tabular} & \begin{tabular}{@{}l@{}} $10919r + 1915$ \end{tabular}\\ \hline \\ 
$2$ & $(0, 2, 8)$ & \begin{tabular}{@{}l@{}} $(2, 1, 2, 2, 2, 2, 2, 2, 2, 1)$ \end{tabular} & \begin{tabular}{@{}l@{}} 47 \end{tabular} & \begin{tabular}{@{}l@{}} $x_{1}^{4}x_{2}x_{3}^{4}x_{4}^{4}x_{5}^{5}x_{6}^{7}x_{7}^{7}x_{8}^{7}x_{9}^{7}x_{10}$ \end{tabular} & \begin{tabular}{@{}l@{}} $-5278r + 34570$ \end{tabular}\\ \hline \\ 
$2$ & $(1, 0, 9)$ & \begin{tabular}{@{}l@{}} $(2, 0, 2, 2, 2, 2, 2, 2, 2, 2)$ \end{tabular} & \begin{tabular}{@{}l@{}} 51 \end{tabular} & \begin{tabular}{@{}l@{}} $x_{1}^{4}x_{3}^{4}x_{4}^{4}x_{5}^{5}x_{6}^{5}x_{7}^{5}x_{8}^{8}x_{9}^{8}x_{10}^{8}$ \end{tabular} & \begin{tabular}{@{}l@{}} $12212r + 3790$ \end{tabular}\\ \hline \\ 
$2$ & $(0, 1, 9)$ & \begin{tabular}{@{}l@{}} $(2, 1, 2, 2, 2, 2, 2, 2, 2, 2)$ \end{tabular} & \begin{tabular}{@{}l@{}} 54 \end{tabular} & \begin{tabular}{@{}l@{}} $x_{1}^{4}x_{3}^{4}x_{4}^{4}x_{5}^{6}x_{6}^{6}x_{7}^{6}x_{8}^{8}x_{9}^{8}x_{10}^{8}$ \end{tabular} & \begin{tabular}{@{}l@{}} $82561r + 56500$ \end{tabular}\\ \hline \\ 
$2$ & $(0, 0, 10)$ & \begin{tabular}{@{}l@{}} $(2, 2, 2, 2, 2, 2, 2, 2, 2, 2)$ \end{tabular} & \begin{tabular}{@{}l@{}} 65 \end{tabular} & \begin{tabular}{@{}l@{}} $x_{1}^{2}x_{2}^{4}x_{3}^{4}x_{4}^{4}x_{5}^{6}x_{6}^{9}x_{7}^{9}x_{8}^{9}x_{9}^{9}x_{10}^{9}$ \end{tabular} & \begin{tabular}{@{}l@{}} $-7028r + 1658$ \end{tabular}\\ \hline \\ 
\end{longtable}

\end{document}